\documentclass[11pt]{amsart}

\usepackage{enumerate,url,amssymb,  mathrsfs, graphicx, pdfsync}

\usepackage{float}
\newtheorem{theorem}{Theorem}[section]
\newtheorem{lemma}[theorem]{Lemma}
\newtheorem*{lemma*}{Lemma}

\theoremstyle{definition}
\newtheorem{definition}[theorem]{Definition}
\newtheorem{example}[theorem]{Example}

\newtheorem{proclaim}[theorem]{\the\prt}

\theoremstyle{remark}

\newtoks\prt
\def\eqn#1$$#2$${\begin{equation}\label#1#2\end{equation}}

\newtoks\by
\newtoks\paper
\newtoks\book
\newtoks\jour
\newtoks\yr
\newtoks\pages
\newtoks\vol
\newtoks\publ
\def\ota{{\hbox\vol{???}}}
\def\cLear{\by=\ota\paper=\ota\book=\ota\jour=\ota\yr=\ota
\pages=\ota\vol=\ota\publ=\ota}
\def\endpaper{\the\by, {\the\paper},
\textit{\the\jour} \textbf{\the\vol} (\the\yr), \the\pages.\cLear}
\def\endbook{\the\by, \textit{\the\book}, \the\publ.\cLear}
\def\endprep{\the\by, \textit{\the\paper}, \the\jour.\cLear}
\def\endyearprep{\the\by, \textit{\the\paper}, \the\jour, (\the\yr).\cLear}
\def\name#1#2{#2 #1}

\numberwithin{equation}{section}
\renewcommand{\bar}[1]{\overline{#1}}

\newcommand{\yy}{\mathbb{Y}}

\newcommand{\rr}{\mathbb{R}}

\newcommand{\abs}[1]{\lvert#1\rvert}

\newcommand{\C}{\mathbb{C}}

\newcommand{\Q}{\mathbb{Q}}

\newcommand{\R}{\mathbb{R}}
\newcommand{\X}{\mathbb{X}}

\newcommand{\Y}{\mathbb{Y}}
\newcommand{\N}{\mathbb{N}}

\newcommand{\dtext}{\textnormal d}

\newcommand{\onto}{\xrightarrow[]{{}_{\!\!\textnormal{onto\,\,}\!\!}}}

\def\en{\mathbb N}
\def\er{\mathbb R}

\DeclareMathOperator{\diam}{diam}

\DeclareMathOperator{\loc}{loc}

\def\leq{\leqslant}
\def\geq{\geqslant}

\def\le{\leqslant}
\def\ge{\geqslant}

\def\XXint#1#2#3{{\setbox0=\hbox{$#1{#2#3}{\int}$}\vcenter{\hbox{$#2#3$}}\kern-.5\wd0}}

\def\XXiint#1#2#3{{\setbox0=\hbox{$#1{#2#3}{\iint}$}\vcenter{\hbox{$#2#3$}}\kern-.5\wd0}}

\begin{document}
\title[3D--Sobolev homeomorphic extensions]{Sobolev homeomorphic extensions\\ from two to three dimensions}

\author[S. Hencl]{Stanislav Hencl}
\address{Charles University, Department of Mathematical Analysis Sokolovsk\'a 83, 186 00 Prague 8, Czech Republic}
\email{hencl@karlin.mff.cuni.cz}

\author[A. Koski]{Aleksis Koski}
\address{Department of Mathematics and Statistics, P.O. Box 68 (Pietari Kalmin katu 5), FI-00014 University of Helsinki, Finland}
\email{aleksis.koski@helsinki.fi}

\author[J. Onninen]{Jani Onninen}
\address{Department of Mathematics, Syracuse University, Syracuse,
NY 13244, USA and  Department of Mathematics and Statistics, P.O.Box 35 (MaD) FI-40014 University of Jyv\"askyl\"a, Finland
}
\email{jkonnine@syr.edu}
\thanks{S. Hencl was supported by the grant GA\v CR P201/21-01976S  A. Koski was supported by the Academy of Finland grant number 307023.}

\subjclass[2010]{Primary 46E35, 58E20}


\keywords{Sobolev homeomorphisms, Sobolev extensions, $L^1$-Beurling-Ahlfors extension}

\begin{abstract} We study the basic question of characterizing which boundary homeomorphisms of the unit sphere can be extended to a Sobolev homeomorphism of the interior in 3D space.  While the planar variants of this problem are well-understood, completely new and direct ways of constructing an extension are required in 3D. We prove, among other things,  that a Sobolev homeomorphism $\varphi \colon \R^2 \onto \R^2$ in $W_{\loc}^{1,p} (\R^2 , \R^2)$ for some $p\in [1,\infty )$ admits a homeomorphic extension $h \colon \R^3 \onto \R^3$ in $W_{\loc}^{1,q} (\R^3, \R^3) $ for $1\le q < \frac{3}{2}p$. Such an extension result  is  nearly sharp,
as the bound $q=\frac{3}{2}p$ cannot be improved due to the H\"older embedding. The case $q=3$  gains an additional interest as it also provides  an $L^1$-variant of the celebrated  Beurling-Ahlfors extension result.
\end{abstract}

\maketitle
\section{Introduction}
Throughout this paper $\mathbb B$ denotes the unit ball in $\R^3$ and $\mathbb S = \partial \mathbb B$.
We study the following {\it 3D--Sobolev homeomorphic extension problem}.
\vskip0.2cm
\begin{qu}  {\it Suppose that  a homeomorphism $\varphi \colon  \mathbb S  \onto  \mathbb S$ admits a continuous extension to $\mathbb B$ in the Sobolev space $W^{1,q} (\mathbb B , \R^3)$ for some $q\in [1, \infty)$. Does the map $\varphi$ also admit a homeomorphic extension to $\mathbb B$ of  class  $W^{1,q} (\mathbb B , \R^3)$?
}
\end{qu}
\vskip0.2cm

Every boundary homeomorphism $\varphi \colon  \mathbb S \onto  \mathbb S $ extends as a homeomorphism to the ball $\mathbb B$.
On the other hand, according to a famous result of  Gagliardo~\cite{Ga},  for $1<q<\infty$,   the  mapping $\varphi$ is the Sobolev trace of some (possibly non-homeomorphic) mapping in $W^{1,q} (\mathbb B , \R^3)$  if and only if  it belongs to the fractional Sobolev space $W^{1-\frac{1}{q}, q} (\mathbb S , \R^3)$; that is,
\begin{equation}\label{eq:contextensioncondition}
\int_{ \mathbb S} \int_{ \mathbb S} \frac{\abs{\varphi (x) - \varphi (y)}^q}{\abs{x-y}^{q+1}} \, \dtext x \, \dtext y  < \infty\, . 
\end{equation}
Note that the 2D  result~\cite{Ve} that  every boundary homeomorphism  $\varphi \colon \partial \mathbb D \onto \partial \mathbb D$  extends as a $W^{1,q}$-homeomorphism, $q<2$, to the unit disk $\mathbb D \subset \R^2$ has no counterpart in higher dimensions. Indeed, there are boundary homeomorphisms from $\mathbb S$ onto itself that do not even admit a continuous Sobolev extension in $W^{1,q} (\mathbb B , \R^3)$ for any $q>1$, see Example~\ref{ex:1}.

First we give a discrete variant of~\eqref{eq:contextensioncondition}; that is, we  characterize the  boundary homeomorphisms that admit a Sobolev extension in $W^{1,q} (\mathbb B , \R^3)$ when $q>2$.  
\begin{theorem}\label{thm:discretechara}
Let $\varphi \colon  \mathbb S \onto  \mathbb S$ be a homeomorphism and $q \in (2, \infty)$. Suppose that $\tilde{\mathcal{D}}_k$  is a  dyadic decomposition of $ \mathbb S$ into closed bi-Lipschitz squares of diameter $c2^{-k}$.  Then $\varphi$ satisfies~\eqref{eq:contextensioncondition} if and only if
\begin{equation}\label{eq:discretecontextension}
\sum_{k=1}^\infty   2^{k (q-3)}  \sum_{\tilde{Q}_j \in \tilde{\mathcal{D}}_k}   \big[ \diam \varphi (\tilde{Q}_j) \big]^q < \infty \,  .
\end{equation}
\end{theorem}
For the precise definition of $\tilde{\mathcal{D}}_k$ we refer to Definition~\ref{def:dyadicsphere}.

The corresponding 2D--Sobolev homeomorphic extension problem~\cite{KOext} has an easy answer thanks to the available analytic methods of  constructing 2D-Sobolev homeomorphisms. Indeed, let $\mathbb D$ be the unit disk in $\mathbb R^2$ and $q\in [1, \infty)$ then a boundary homeomorphism $\varphi \colon \partial \mathbb D \onto \partial \mathbb D$ admits a homeomorphic extension to $\overline{\mathbb D}$ in $W^{1,q} (\mathbb D , \R^2)$ if and only if it admits a continuous extension to  $\overline{\mathbb D}$ in $W^{1,q} (\mathbb D , \R^2)$. This follows from the Rad\'o-Kneser-Choquet (RKC) theorem~\cite{Dub} for  $q\le 2$.  The RKC theorem asserts that a
homeomorphic boundary value $\varphi \colon \partial \mathbb D \onto \partial \mathbb D$  admits a homeomorphic harmonic extension of $\mathbb D$. The harmonic extension  belongs to $W^{1,q} (\mathbb D, \R^2)$ for all $q<2$ and to  $W^{1,2} (\mathbb D, \R^2)$ exactly when is in the trace space of   $W^{1,2} (\mathbb D, \R^2)$.  Similarly the $q$-harmonic variants of the RKC theorem~\cite{AS} solve the 2D extension problem for $q > 2$. An analogous approach  fails in higher dimensions. Indeed, Laugesen~\cite{La} constructed a self-homeomorphism of the sphere  $\mathbb S$ in $\R^3 $ whose harmonic extension
to the ball $\mathbb B$ is not injective. Thus, the 3D extension problem requires new methods of constructing Sobolev homeomorphisms.

Our main result tells us that the searched homeomorphic extension exists if the boundary homeomorphism satisfies a strengthened version of the condition~\eqref{eq:discretecontextension}.
\begin{theorem}\label{thm:mainsum}
Let $q \in (1, \infty)$.  Suppose that $\tilde{\mathcal{D}}_k$  is a dyadic decomposition of $ \mathbb S$ into closed bi-Lipschitz squares of diameter $c2^{-k}$.   If  a homeomorphism $\varphi \colon  \mathbb S \onto  \mathbb S$  satisfies 
\begin{equation}\label{eq:discretehomeoextension}
\sum_{k=1}^\infty   2^{k (q-3)}  \sum_{\tilde{Q}_j \in  \tilde{\mathcal{D}}_k}   \big[ \mathcal H^1 \big( \varphi (\partial \tilde{Q}_j) \big) \big]^q < \infty \,  ,
\end{equation}
then it admits a homeomorphic extension $h \colon \overline{\mathbb B} \onto \overline{\mathbb B}$ in $W^{1,q} (\mathbb B , \R^3)$.
\end{theorem}
Here $\mathcal H^1$  stands for  $1$-dimensional Hausdorff measure and so $\mathcal H^1 \big( \varphi (\partial \tilde{Q}_j) \big) $ measures the length of the curve $\varphi (\partial \tilde{Q}_j)$.

For a   Sobolev homeomorphism  $\varphi \colon \mathbb S \onto \mathbb S$  the trivial radial extension $h(x)= \abs{x} \varphi (x)$ produces a self homeomorphism of $\overline{\mathbb B}$ which has the same Sobolev regularity as the given boundary map $\varphi$. Clearly, such an extension is far from being optimal. Our next result, however, nearly characterizers the first order Sobolev spaces that admit a Sobolev homeomorphic extension to $\mathbb B$.
\begin{theorem}\label{thm:mainsobo}
Let $\varphi \colon \mathbb S \onto \mathbb S$ be a homeomorphism in $W^{1, p} (\mathbb S , \R^3)$ for some $p \in [1, \infty)$. Then $\varphi$ admits a homeomorphic extension $h \colon \overline{\mathbb B} \onto \overline{\mathbb B}$ in $W^{1,q} (\mathbb B , \R^3)$ for $1 \le q < \frac{3}{2} p$. 
\end{theorem}
For the sharpness of this result  we refer to the general embedding result by Sickel and Triebel~\cite[Theorem 3.2.1]{ST}. Namely for $p \in (1, \infty)$ we have $W^{1, p} (\mathbb S , \R^3) \subset W^{1-\frac{1}{q}, q} (\mathbb S , \R^3)$ if and only if $q \le  \frac{3}{2} p$. Even assuming that the mappings are homeomorphisms does not improve the inclusion at least when $p \ge 2$, see Example~\ref{ex:sobohomeo}. We do not know if one can take $q= \frac{3}{2} p$  in Theorem~\ref{thm:mainsobo}.  

Theorem~\ref{thm:mainsobo} follows from Theorem~\ref{thm:mainsum}.  On the contrary there are self homeomorphisms of $\mathbb S$ which satisfy~\eqref{eq:discretehomeoextension} and do not belong to any Sobolev class $W^{1,p} (\mathbb S, \R^3)$, $p\ge 1$, see Example~\ref{ex3}.

In topology and analysis, a number of extension problems have been studied.
A demand for  Sobolev homeomorphic extension problems comes from the variational approach to  Geometric Function Theory (GFT)~\cite{AIMb,  HKb, IMb,  Reb} and  mathematical models of Nonlinear Elasticity (NE)~\cite{Anb, Bac, Cib}. Both theories  enquire into homeomorphisms $h  \colon  \mathbb X \onto \mathbb Y $ of smallest \textit{stored energy}
\[
\mathsf E_\mathbb X [h] =  \int_\mathbb X \mathbf {\bf E}(x,h, Dh )\,  \dtext x\, ,   \qquad  \mathbf E \colon  \mathbb X \times \mathbb Y \times \mathbb R^{n\times n} 
\]
where   the so-called \textit{stored energy function} $\mathbf E$ characterizes the mechanical and elastic properties of the material occupying the domains. In a pure displacement setting, typically an orientation-preserving  boundary homeomorphism $\varphi \colon \partial \X \onto \partial \Y$ is given. The class of admissible deformations consists of  Sobolev homeomorphisms or just  Sobolev mappings $h \colon \overline{\X} \onto \overline{\Y}$ with non-negative Jacobian determinant $J_h (x) = \det Dh(x) \ge0$ (an axiomatic assumption in NE) which coincides with $\varphi$ on the boundary and having a finite stored energy.  In such  variational problems,  a first issue to address is  the non-emptiness  of the class of admissible deformations; that is, to solve the corresponding Sobolev homeomorphic extension problem.

 Note that an arbitrary  orientation-preserving  Sobolev homeomorphism $h$ need not be {\it strictly orientation-preserving} in the sense that $J_h (x) = \det Dh(x) > 0$ almost everywhere. For every $q<3$, there even exists a homeomorphism $h \colon \mathbb B \onto \mathbb B$ in $W^{1,q} (\mathbb B , \R^3)$  with $J_h(x) =0$ for almost every $x\in \mathbb B$, see~\cite{He0}. However, the homeomorphic extensions $h \colon \mathbb B \onto \mathbb B$ construtced in  Theorem~\ref{thm:mainsobo} and Theorem~\ref{thm:mainsum} are piecewise linear.  Thus, they are strictly orientation-preserving  provided that  the given  boundary homeomorphism itself  preserves  the orientation.  In particular, these homeomorphisms have   finite distortion.  The theory of mappings of finite distortion arose out of a need to extend the ideas and applications of the classical theory of quasiconformal mappings to the degenerate elliptic setting~\cite{HKb, IMb}.   We recall  that  a homeomorphism $h \colon \mathbb X \onto \mathbb Y$ of Sobolev class $W^{1,1}_{\loc}(\mathbb X, \mathbb R^n)$ defined on a domain $\X\subset \R^n$ has \textit{finite distortion} if
\begin{equation}\label{eq:dist}
 \abs{Dh(x)}^n \le K(x) J_h(x) 
\end{equation}
for some measurable function $1 \le K(x) < \infty$. Here, $\abs{Dh(x)}$ is the operator norm of the weak differential $Dh(x) \colon \X \to \R^n$ of $h$ at a point $x \in \X$.
We obtain  {\it quasiconformal} mappings if $K \in  L^\infty (\X)$.  There are several other distortion functions of great
interest in GFT. 
 Each of them is designed to measure the deviation from conformality of a given
mapping $h\colon  \mathbb X \to \R^n $ in terms of the tangent  
linear  map $Dh(x) \colon \R^n \to \R^n$. The most
interesting, from the applied point of view, is the inner distortion
function. In NE  one is typically  provided information not
only on   the differential matrix, but also on its  $(n-1)\times (n-1) $--minors; that is, the {\it cofactor matrix} $D^\sharp h$  called {\it co-differential} of $h$.  Now, for a homemorphism
 $h\in W^{1,1}_{\loc}(\mathbb X, \R^n)$
  of finite distortion we introduce its
  \textit{inner distortion} function, to be the smallest
  $K_{_I} (x) = K_{_I} (x,f ) \ge 1$  satisfying  
  \[  \abs{D^\sharp  f(x)} ^n =   K_{_I}(x) \cdot  J_f(x)^{n-1} \]

  The most pronounced extension result  in GFT is  the   Beurling-Ahlfors quasiconformal extension theorem~\cite{BA}.  It  states that a self-homeomorphism of the unit disk $\mathbb D$   is quasiconformal if and only if the boundary correspondence  homeomorphism $\varphi \colon \partial \mathbb D \onto  \partial \mathbb D$ is quasisymmetric. 
The   Beurling-Ahlfors result has found a number of  applications in Teichm\"uller theory, Kleinian groups, conformal welding
and dynamics, see e.g.~\cite{AIMb, Hu}.  It has  generalized to  the $n$-dimensional quasiconformal maps as well, first for $n= 3$ by Ahlfors~\cite{Ah}  and then for $n= 4$ by   Carleson~\cite{Ca}.  A full $n$-dimensional version of the Beurling-Ahlfors extension is due to Tukia and V\"ais\"al\"a~\cite{TV}. 
Their extension  uses,  among other things, Sullivan's theory~\cite{Su} of deformations of Lipschitz embeddings.
Moreover, Astala, Iwaniec, Martin and Onninen~\cite{AIMO}, as a part of their studies
of deformations with smallest mean distortion, characterizes  self homeomorphisms of the unit circle that admit a homeomorphic extension to the unit disk $\mathbb D$ with integrable distortion.  This $L^1$--Beurling-Ahlfors extension theorem enjoys the following 3D-variant.
\begin{theorem}\label{thm:corollary}
Let $\psi \colon \mathbb S \onto \mathbb S$ be an orientation-preserving homeomorphism. Suppose that the inverse  $\psi^{-1}=\varphi$ satisfies~\eqref{eq:discretehomeoextension} with $q=3$. Then  $\psi$ admits a homeomorphic extension $f \colon \overline{\mathbb B} \onto \overline{\mathbb B}$ with integrable inner distortion.
\end{theorem}
Theorem~\ref{thm:corollary} is  actually  a relatively straightforward  consequence of  Theorem~\ref{thm:mainsum}, thanks to an important connection between the conformal energy of a homeomorphism and the inner distortion function of the inverse mapping.  Indeed   it is easy to see, at least formally, that the pullback of the $3$-form $K_{_I}(y, f) \,   \dtext y \in \wedge ^3\mathbb B$ by the inverse mapping
  $ f^{-1} \colon  \mathbb B \onto \mathbb B$ is equal to  $\abs{Df^{-1}(x)}^3 \, \dtext x \in \wedge ^3\mathbb B$. This observation is the key to
the  identity,
 \begin{equation}\label{eq:identity}
 \int_{\mathbb B} \abs{Dh (x)}^3 \, \dtext x = \int_{\mathbb B} K_{_I}(y, f) \, \dtext y \, , \quad \textnormal{ where } h=f^{-1} \colon \mathbb B \onto \mathbb B \, . 
 \end{equation} 
 The optimal Sobolev regularity of deformations to guarantee  the identity is well-understood today, \cite{CHM, HKM, HKO, Onreg}. In particular, if a homeomorphism $h \colon \mathbb B \onto \mathbb B$ of finite distortion belongs to the Sobolev class $W^{1,3} (\mathbb B , \R^3)$, then  the inverse $f=h^{-1}$ has integrable inner distortion. Thus, Theorem~\ref{thm:corollary}  simply follows from  Theorem~\ref{thm:mainsum}.
It is worth noting that  the borderline case in Theorem~\ref{thm:mainsobo} 
(p = 3 and q = 2), if true, would have an interesting corollary. Namely, a homeomorphism
$\psi \colon \R^2 \onto \R^2$ of locally integrable distortion would then admit a homeomorphic extension
$f \colon \R^3 \onto \R^3$ with locally integrable inner distortion.

\section{A discrete characterization, proof of Theorem~\ref{thm:discretechara}}
Let $\mathbb I=[a,b]^2$ be an initial square  in $\R^2$. The standard {\it dyadic decomposition} of $\mathbb I$ consists of closed squares $\tilde{Q} \subset \mathbb I$
with sides  parallel to the sides of $\mathbb I$ and of side length $l (\tilde{Q})= 2^{-k} (b-a)$, $k=1, 2, 3, \dots$; refers to  the {\it $k$-th generation} in the construction.  That is, the squares in the $k$-th  generation have the form
\[\tilde{Q}_j=2^{-k} (\mathbb I +v_j ) \subset \mathbb I \, , \qquad \textnormal{for some } v_j \in \mathbb R^2 \, . \]
They cover $\mathbb I$ and  have  side length  $2^{-k} (b-a)$. The collection of  the $k$-th generation squares are denoted by $\tilde{\mathcal{D}}_k$.  There are $2^{2k}$  squares in $\tilde{\mathcal{D}}_k$. The interiors  of the squares in the same generation $\tilde{\mathcal{D}}_k$ are  pairwise disjoint.

Let $\mathbb Q^3 =[0,1]^3 $ be the unit cube in $\R^3$. We define the $k$-th generation  dyadic decomposition of  $\partial \mathbb Q^3$ as follows:  first we divide each of the six faces of $\partial \mathbb Q$ into the $k$-th generation   squares and then  the $k$-th generation  dyadic decomposition of  $\partial \mathbb Q^3$ simply consists  of  the union of these closed squares. 

Now, since $\overline{\mathbb B}$ is a bi-Lipschitz equivalent with $\mathbb Q^3$, defining a $k$-th generation  dyadic decomposition of $\partial \mathbb B = \mathbb S$ can be easily induced from the above case.
\begin{definition}\label{def:dyadicsphere} Let $\Phi \colon \R^3 \to \R^3$ be a bi-Lipschitz map which takes  $\mathbb Q^3$ onto $\overline{\mathbb B}$. Then the  $k$-th generation dyadic decomposition of $\mathbb S$, denoted by $\tilde{\mathcal{D}}_k$, consists of $\Phi (\tilde{Q}_j)$, where $\tilde{Q}_j$ is a $k$-th generation dyadic square of $\partial \mathbb Q^3$.
\end{definition}

\begin{theorem}\label{thm:discretecharaauxi}
Let $\varphi \colon \R^2 \to \R^2$ be a homeomorphism, $\mathbb I_R= [-R,R]^2 \subset \R^2$ for $R>0$ and $N\in \mathbb N$. Denote the collection of  $k$-th generation dyadic squares of   $ \mathbb I_N$ by  $\tilde{\mathcal{D}}^N_k$. Then, for $2<q< \infty$ we have
\begin{equation}\label{eq:intcond}
\int_{ \mathbb I_R} \int_{ \mathbb I_R} \frac{\abs{\varphi (x) - \varphi (y)}^q}{\abs{x-y}^{q+1}} \, \dtext x \, \dtext y  < \infty   \quad \textnormal{for every } R>0 
\end{equation}
if and only if
\begin{equation}\label{eq:sumcond}
\sum_{k=1}^\infty   2^{k (q-3)}  \sum_{\tilde{Q}_{j} \in \tilde{\mathcal{D}}^N_k}   \big[ \diam \varphi (\tilde{Q}_j) \big]^q < \infty   \quad \textnormal{for every } N\in \mathbb N \,  .
\end{equation}
\end{theorem}
\begin{proof}
First we assume the condition~\eqref{eq:intcond} with $R=2^{12}$. Now,  the mapping  $\varphi \colon \R^2 \to \R^2$ admits a continuous extension $f \colon \R^3 \to \R^3$ in $W^{1,p} ( \mathbb I_{R} \times [-R,R] , \R^3)$ (see \eqref{eq:contextensioncondition} and the paragraph before). It suffices to prove~\eqref{eq:sumcond} with $N=1$. 

Fix $\tilde{Q}_{k,j} \in \tilde{\mathcal{D}}^1_k$ for some  $k\in \mathbb N$ and $j \in \{1, \dots , 2^{2k}\}$. We denote the center of $\tilde{Q}_{k,j} \subset \R^2$ by $x_\circ$.  Let $\mathbb B^3_R$ be the 3-dimensional ball in $\R^3$ centered at $x_\circ$ with radius $R>0$ and
\begin{equation}\label{eq:triv}
\mathbb B^2_{R} = \mathbb B^3_R \cap (\mathbb R^2 \times \{0 \}) \, . 
\end{equation}

 Choose $\eta\in (2,q)$.   According to the Sobolev imbedding theorem on spheres~\cite[Lemma 2.20]{HKb}  there is a constant $C>0$ such that for a.e. $s\in (0, R)$ we have 
\[\diam f(\partial \mathbb B^3_s)  \le C \,  s^{1-\frac{2}{\eta}} \left( \int_{\partial \mathbb B^3_s} \abs{Df}^\eta \right)^\frac{1}{\eta}  \, . \]
This is the moment where  we used the assumption  $q>2$.  By~\eqref{eq:triv} we always have
\[ \diam f(\partial \mathbb B^2_s) \le \diam f(\partial \mathbb B^3_s) \, . \]
Since $\varphi \colon \R^2 \onto \R^2$ is a homeomorphism we get
\[   \diam \varphi( \mathbb B^2_s) = \diam \varphi(\partial \mathbb B^2_s) \, . \]
For fixed $r\in (0,R/2)$,  the above estimates give
\[  \diam \varphi( \mathbb B^2_r) \le    C\,  s^{1-\frac{2}{\eta}} \left( \int_{\partial \mathbb B^3_s} \abs{Df}^\eta \right)^\frac{1}{\eta}  \qquad \textnormal{for a.e. } s\in (r,R)\ \]
and
\begin{equation}\label{eq:sobospheres}  \left[\diam \varphi( \mathbb B^2_r) \right]^\eta \int_r^{2r} \frac{\dtext s}{s^{\eta-2}}  \le C \int_{\mathbb B^3_{2r} \setminus \mathbb B^3_r} \abs{Df}^\eta \, . 
\end{equation}
Thus 
\[ \diam \varphi( \mathbb B^2_r) \le C r^{1-\frac{3}{\eta}}   \left( \int_{\mathbb B^3_{2r}} \abs{Df}^\eta \right)^\frac{1}{\eta}  \]
and 
\begin{equation}\label{eq:diamestimate}
 \diam \varphi(\tilde{Q}_{k,j}) \le C 2^{-k(1-3/\eta)}    \left( \int_{\mathbb B^3_{2^{3-k}}} \abs{Df}^\eta \right)^\frac{1}{\eta}  \, .
 \end{equation}
The $k$-th dyadic decomposition $\tilde{\mathcal{D}}_k = \{\tilde{Q}_{k,j} \colon k\in \mathbb N  \, , \; j=1, \dots , 2^{2k}  \}$ of $\mathbb I_1 \subset \R^2$ defines a corresponding Whitney decomposition of $\mathbb I_1 \times [0,2] \subset \R^3$,
\[  \mathcal W_k =  \{\tilde{Q}^3_{k,j} \colon k\in \mathbb N  \, , \; j=1, \dots , 2^{2k}  \}  \]
where
\[\tilde{Q}^3_{k,j}  = \tilde{Q}_{k,j}  \times [2^{-k+1}, 2^{-k+2}] \, . \]
Let $x\in \tilde{Q}^3_{k,j} $ and $c=2^{11}$. Then $B^3_{c2^{-k}}(x) = B^3 (x, {c2^{-k}}) \supset \mathbb B_{2^{3-k}}^3$ and so
\[   \diam \varphi(\tilde{Q}_{k,j}) \le C 2^{-k(1-3/\eta)}    \left( \int_{ B^3_{c2^{-k}  } (x)} \abs{Df}^\eta \right)^\frac{1}{\eta}   \]
by~\eqref{eq:diamestimate}. In particular, we have
\begin{equation}\label{eq:maximalestimate}
\diam \varphi(\tilde{Q}_{k,j}) \le C 2^{-k}   \big[ {\bf M}_{c}  \abs{Df}^\eta (x)\big]^\frac{1}{\eta}  \qquad \textnormal{ for all } x\in  \tilde{Q}^3_{k,j} \, .
 \end{equation}
Here ${\bf M}_c$ denotes the Hardy-Littlewood maximal operator,
\[  {\bf M}_c \abs{Df}^\eta  (x) = \sup_{r<c}  \frac{1}{\abs{ B_r^3 (x) }} \int_{ B_r^3 (x)}\abs{Df}^\eta    \, .   \]
Raising the estimate~\eqref{eq:maximalestimate} to the power $q$ and then integrating it over the cube $\tilde{Q}^3_{k,j}$ we have
\[2^{-3k}   \big[ \diam \varphi(\tilde{Q}_{k,j}) \big]^q \le C 2^{-qk}   \int_{\tilde{Q}^3_{k,j}}    \big[ {\bf M}_{c}  \abs{Df}^\eta (x)\big]^\frac{q}{\eta} \, .\]
Thus,
\[
\begin{split}
 \sum_{k=1}^\infty \sum_{j=1}^{2^{2k}} 2^{k(q-3)}  \big[ \diam \varphi(\tilde{Q}_{k,j}) \big]^q   &\le C  \sum_{k=1}^\infty \sum_{j=1}^{2^{2+2k}}   \int_{Q^3_{k,j}}    \big[ {\bf M}_{c}  \abs{Df}^\eta (x)\big]^\frac{q}{\eta}\\
 &  =  C   \int_{ \mathbb I_1 \times [0,2] }    \big[ {\bf M}_{c}  \abs{Df}^\eta (x)\big]^\frac{q}{\eta} .
 \end{split} \]
Since $q/\eta>1$ we can use the boundedness of the Hardy-Littlewood maximal function in $L^{\frac{q}{\eta}}$ for function $|Df|^\eta$ to obtain
\[  \sum_{k=1}^\infty \sum_{j=1}^{2^{2k}} 2^{k(q-3)} \big[ \diam \varphi(\tilde{Q}_{k,j}) \big]^q   \le C  \int_{ \mathbb I_{c} \times [-2c,2c] } \abs{Df}^q  \,  \]
as claimed.
\vskip0.2cm

Second we assume~\eqref{eq:sumcond} for $N=1$ and some $q\in (1, \infty)$. Our goal is show that
\[\int_{ \mathbb I_1} \int_{ \mathbb I_1} \frac{\abs{\varphi (x) - \varphi (y)}^q}{\abs{x-y}^{q+1}} \, \dtext x \, \dtext y  < \infty \, . \]

We say that two dyadic squares on the same level $k$ are \emph{neighbors} if their boundaries have at least one intersection point. We also define the \emph{dyadic distance} $d^*(S,S')$ of two squares $S, S' \in \tilde{\mathcal{D}}^1_k$   as the number of neighbors one has to travel through to reach $S'$ from $S$, so that two dyadic neighbors themselves have a distance of $0$. If $S,S'$ are such squares then we denote $S\vert S'$ if the dyadic distance between $S$ and $S'$ is either $1$ or $2$. We first note that
\begin{equation}\label{eq:dya}\int_{ \mathbb I_1} \int_{ \mathbb I_1} \frac{\abs{\varphi (x) - \varphi (y)}^q}{\abs{x-y}^{q+1}} \, \dtext x \, \dtext y  \leq \sum_{S \vert S'} \int_{ S}\int_{ S'}  \frac{\abs{\varphi (x) - \varphi (y)}^q}{\abs{x-y}^{q+1}} \, \dtext x \, \dtext y \end{equation}
where the sum is taken over all levels of dyadic squares and all pairs for which $S\vert S'$ holds. This is due to the geometric fact that for every pair of points $x,y \in \mathbb I_1$ there are dyadic squares with $S\vert S'$ so that $x \in S$ and $y \in S'$.

Let now $S\vert S'$ with $x \in S \in \tilde{\mathcal{D}}^1_k$ and $y \in S' \in \tilde{\mathcal{D}}^1_k$. Denote by $S_1\in \tilde{\mathcal{D}}^1_k$ and $S_2\in \tilde{\mathcal{D}}^1_k$ two different dyadic squares  so that $(S,S_1,S_2,S')$ form a sequence of dyadic squares for which each successive pair is a neighbor. Then we simply estimate that
\begin{align*}|\varphi(x) - \varphi(y)| &\leq \diam \varphi (S) + \diam \varphi (S_1) + \diam \varphi (S_2) + \diam \varphi (S')
\\&\leq \sum_{d^*(S,\tilde{Q}) \leq 2} \diam \varphi (\tilde{Q}).
\end{align*}
Note that the sum in the last expression has at most $49$ terms. Hence if we sum this expression over all dyadic squares $S$, every dyadic square will be repeated at most $49$ times.
Plugging this into \eqref{eq:dya}  and using \eqref{eq:sumcond} gives
\begin{align*}
\int_{ \mathbb I_1} \int_{ \mathbb I_1} \frac{\abs{\varphi (x) - \varphi (y)}^q}{\abs{x-y}^{q+1}} \, \dtext x \, \dtext y
&\leq \sum_{k=1}^\infty \sum_{S \in \tilde{\mathcal{D}}^1_k} \int_{ S} \int_{ S'} \frac{49^p \left[\diam \varphi (S)\right]^q}{2^{-(q+1)k}}  \, \dtext x \, \dtext y
\\&\leq 49^q\sum_{k=1}^\infty  \frac{2^{-2k} 2^{-2k}}{2^{-(q+1)k}} \sum_{S \in \tilde{\mathcal{D}}^1_k}   \big[ \diam \varphi (S) \big]^q
\\& < \infty.
\end{align*}
\end{proof}

Clearly, Theorem~\ref{thm:discretechara} is an immediate consequence of Theorem~\ref{thm:discretecharaauxi}.

\section{Examples}



An arbitrary homeomorphism $\varphi \colon \partial \mathbb D \onto \partial \mathbb D$ admits a homeomorphic extension to the unit disk $\mathbb D \subset \R^n$ in the Sobolev class $W^{1,q} (\mathbb D , \R^2)$ for all $q<2$. Our next example shows that such a result has no 3D counterpart.

\begin{example}\label{ex:1}
There is a Sobolev homeomorphism $\varphi \colon {\mathbb S} \onto {\mathbb S}$ such that   $\varphi \not \in W^{1-\frac{1}{q}, q} (\mathbb S , \R^3)$  for any $q>1$ and hence it does not admit a continuous extension $f \colon \overline{\mathbb B} \to \R^3$ in $W^{1,q}(\mathbb B,\R^3)$.
\end{example}


\begin{proof}
We simplify our writing here and   construct a Sobolev homeomorphism $\varphi \colon [0,1]\times [0,1] \onto [0,1] \times [0,2]$ with $\varphi (0,0)= \varphi (1,1)$. Note that this causes no loss of generality  due to a  suitable bilipschitz change of variables  in both  domain and target side.


Let $s \colon \R \to \R$ be a $1$-periodic piecewise linear ``saw'' function defined by
$$
s(x)=\begin{cases}
2x&\text{ for }x\in[0,\frac{1}{2}],\\
2-2x&\text{ for }x\in[\frac{1}{2},1].\\
\end{cases}
$$
We set $s_k(x)=s(x 10^k)$ and obtain a $10^{-k}$-periodic saw function. By induction we choose an 
increasing sequence of integers $n_k$ such that 
\eqn{wewant}
$$
\begin{aligned}
10^{-kq} 10^{(q-1)\frac{1}{2}n_k}&\geq 2^k  \text{ and }\\
\Bigl(\sum_{j=1}^{k-1}10^{-j}\cdot 2\cdot 10^{n_j}\Bigr) 10^{-\frac{1}{2}n_k}&\leq \frac{1}{8}10^{-k}.\\
\end{aligned}
$$
We set 
$$
r_k=10^{-\frac{1}{2}n_k}\text{ and }\phi(x)=\sum_{j=1}^{\infty}10^{-j}s_{n_j}(x).
$$
Note that  $\phi$, being a uniform limit of continuous functions, is also continuous. It is not difficult to check that the mapping $\varphi \colon [0,1]^2 \onto [0,1]\times [0,2]$, defined by
$$
\varphi(x_1,x_2)=[x_1,x_2+\phi(x_1)]\text{ is a homeomorphism}. 
$$
We estimate
\begin{equation}\label{aha}
\begin{split}
& \int_{(0,1)^2\times(0,1)^2}\frac{|\varphi (x)-\varphi (y)|^q}{|x-y|^{q+1}}\; dx \; dy  \\ & \geq 
C \int_{(0,1)^2\times(0,1)^2}\frac{(|\phi(x_1)-\phi(y_1)|-|x_2-y_2|)^q}{|x-y|^{q+1}}\; dx \; dy
\end{split}
\end{equation}
and note that the term $\frac{|x_2-y_2|^q}{|x-y|^{q+1}}\leq\frac{1}{|x-y|}$ in the last integral is integrable. 
Therefore, it suffices to show that the integral
\begin{equation}\label{aha2}
 \int_{(0,1)^2\times(0,1)^2}\frac{|\phi(x_1)-\phi(y_1)|^q}{|x-y|^{q+1}}\; dx \; dy 
\end{equation}
diverges.

For that, let us fix $k\in\en$ and denote 
$$
A_1:=\bigl\{x_1\in[0,1]:\ x_1\in[-\tfrac{1}{8}10^{-n_k}+j 10^{-n_k},\tfrac{1}{8}10^{-n_k}+j 10^{-n_k}]\text{ for }j\in\en\cup\{0\}\bigr\},
$$
 i.e. $s_{n_k}(x_1)\in[0,\tfrac{1}{4}]$ for every $x_1\in A_1$ 
and 
$$
A_2=\{y_1\in [0,1]:\ y_1\in [\tfrac{3}{8}10^{-n_k}+j 10^{-n_k},\tfrac{5}{8}10^{-n_k}+j 10^{-n_k}]\text{ for }j\in\en\cup\{0\}\},$$
 i.e. $s_{n_k}(y_1)\in[\tfrac{3}{4},1]$ for every $y_1\in A_2$. 
Given $x_1\in A_1$ we set 
$$
A_2(x_1)=A_2\cap (x_1-r_k,x_1+r_k).
$$  
It is easy to see that for every $x_1\in A_1$ and $y_1\in A_2$ we have 
$$
10^{-k}|s_{n_k}(x_1)-s_{n_k}(y_1)|\geq \frac{1}{2}10^{-k}. 
$$
Further for every $x_1$ and $y_1$ we have 
$$
\Bigl|\sum_{j=k+1}^{\infty}10^{-j}s_{n_j}(x_1)-\sum_{j=k+1}^{\infty}10^{-j}s_{n_j}(y_1)\Bigr|\leq \sum_{j=k+1}^{\infty}10^{-j}\leq \frac{1}{8}10^{-k}. 
$$
The function $10^{-j}s_{n_j}$ is Lipschitz with Lipschitz constant $10^{-j}\frac{1}{10^{-n_j}/2}$. Hence in view of \eqref{wewant}, for every $x_1$ and $y_1$ with $|x_1-y_1|<r_k$ we have 
$$
\Bigl|\sum_{j=1}^{k-1}10^{-j}s_{n_j}(x_1)-\sum_{j=1}^{k-1}10^{-j}s_{n_j}(y_1)\Bigr|\leq \sum_{j=1}^{k-1}10^{-j}\cdot 2\cdot 10^{n_j}\cdot |x_1-y_1|\leq \frac{1}{8}10^{-k}. 
$$
It follows that for every $x_1\in A_1$ and $y_1\in A_2$ with $|x_1-y_1|<r_k$ we have 
$$
\begin{aligned}
|\phi(x_1)-\phi(y_1)|\geq &10^{-k}|s_{n_k}(x_1)-s_{n_k}(x_2)|\\ &-\Bigl|\sum_{j=k+1}^{\infty}10^{-j}s_{n_j}(x_1)-\sum_{j=k+1}^{\infty}10^{-j}s_{n_j}(y_1)\Bigr|
\\ &-\Bigl|\sum_{j=1}^{k-1}10^{-j}s_{n_j}(x_1)-\sum_{j=1}^{k-1}10^{-j}s_{n_j}(y_1)\Bigr|\\
\geq& \frac{1}{4}10^{-k}. 
\end{aligned}
$$

To show \eqref{aha2} we  estimate the integral
$$
C\int_{A_1}\int_{A_2(x_1)}
\int_0^1\int_0^1  \frac{10^{-kq}}{\bigl(|x_1-y_1|+|x_2-y_2|\bigr)^{q+1}}\; dx_2 \; dy_2\; dy_1\; dx_1.
$$
Since applying a change of variables $s=x_2-y_2$ and $t=x_2+y_2$ we  obtain
\[
\begin{split}
\int_0^1\int_0^1  \frac{1}{\bigl(|a|+|x_2-y_2|\bigr)^{q+1}}\; dx_2 \; dy_2
& \geq C\int_{\frac{1}{2}}^{\frac{3}{2}}1\; dt\int_{-\frac{1}{2}}^{\frac{1}{2}}\frac{1}{\bigl(|a|+|s|\bigr)^{q+1}}\; ds \\
& \geq C\frac{1}{|a|^{q}} 
\end{split}
\]
we may estimate \eqref{aha2} from below by the integral 
\begin{equation}\label{aha3}
C\int_{A_1}\int_{A_2(x_1)}\frac{10^{-kq}}{|x_1-y_1|^{q}}\; dy_1\; dx_1.
\end{equation}
We use again a change of variables  $s=x_1-y_1$ and $t=x_1+y_1$. 
Since $|A_1|\geq \frac{1}{4}$ and $|A_2|\geq \frac{1}{4}$ it is not difficult to see that the sets $A_1+A_2$ and $A_1-A_2$ are large enough, i.e. they occupy a large percentage of each interval of size much bigger than $10^{-n_k}$. Together with the fact that 
$r_k=10^{-\frac{1}{2}n_k}$ is much bigger than the period of $s_{n_k}$ which is $10^{-n_k}$ we may estimate the  integral~\eqref{aha3} from below as
$$
C\int_{r_k/2}^{r_k}\frac{10^{-kq}}{|s|^{q}}\; ds\geq C\frac{10^{-kq}}{r^{q-1}_k}.
$$
By \eqref{wewant} we finally conclude that the integral~\eqref{aha2} diverges as we wanted.
\end{proof}

The following example shows the sharpness of Theorem~\ref{thm:mainsobo}. 
\begin{example}\label{ex:sobohomeo}
Let $p \ge 2$ and $q>\frac{3}{2} p$. There is a Sobolev homeomorphism $\varphi \colon {\mathbb S} \onto {\mathbb S}$ such that   $\varphi\in W^{1,p} (\mathbb S, \R^3)$ but  $\varphi\notin W^{1-\frac{1}{q}, q} (\mathbb S , \R^3)$. Hence such a $\varphi$ does not admit a continuous extension $h \colon \overline{\mathbb B} \to \R^3 $ in the Sobolev class $W^{1,q}(\mathbb B,\R^3)$.
\end{example}
\begin{proof}
For simplicity we give a formula for  $\varphi$ from $\mathbb D$ onto itself and not from $\mathbb S$ onto $\mathbb S$.  It is clear that this causes no loss of generality due to a suitable  bilipschitz change of variables. Given our $p\geq 2$ and $q>\frac{3}{2} p$ we choose $\alpha>0$ such that
$$
1-\frac{2}{p} <\alpha<1-\frac{3}{q}.
$$
We set
$$
\varphi(x)=\frac{x}{|x|}|x|^{\alpha}.
$$
A simple computation gives that $\varphi \in W^{1,p} (\mathbb D, \R^2)$. Either by a direct computation we also obtain that $\varphi\notin W^{1-\frac{1}{q}, q} (\mathbb D, \R^2)$ (see e.g. \cite[Lemma 1, page 44]{RS}) or assuming by contradiction that  $\varphi\in W^{1-\frac{1}{q}, q} (\mathbb D, \R^2)$. In the latter case  $\varphi$ admits a continuous extension $h \colon \mathbb D \times (-1,1) \to \R^3$ in  the Sobolev class $W^{1,q} (\mathbb D \times (-1,1) , \R^3)$.  In particular, $h$ is locally $(1-\frac{3}{q})$-H\"older continuous but this is impossible because $h=\varphi$ on $\mathbb D \times \{0\}$ is just   $(1-\frac{2}{\alpha})$-H\"older continuous.
\end{proof}

Theorem~\ref{thm:mainsobo} follows from Theorem~\ref{thm:mainsum}. In the following example we show that on the contrary there is  a homeomorphism $\varphi \colon \mathbb S \onto \mathbb S$ which satisfy the condition~\eqref{eq:discretehomeoextension} in Theorem~~\ref{thm:mainsobo}  and does not belong to any Sobolev class $W^{1,p} (\mathbb S, \R^3)$, $p\ge 1$. Again, we define  $\varphi$  only on $[0,1]^2$, and a  bilipschitz change of variables easily generalizes this homeomorphism  from $\mathbb S$ onto $\mathbb S$.

\prt{Example}
\begin{proclaim}\label{ex3}
Consider 
\eqn{defh}
$$
\varphi(x,y)=[g(x),y]\text{ where }g(x)=x+C(x) 
$$
and $C$ is Cantor function. Not the standard  $1/3$ Cantor function, but $1/k$ Cantor function (for $k\geq 2$), i.e. in each step we remove the middle $1/k$-part of the interval. 
It is not difficult to show that this Cantor function is H\"older continuous with exponent $\alpha=\frac{\log\frac{1}{2}}{\log(\frac{1}{2}(1-\frac{1}{k}))}$. Let us note that 
$$
\lim_{k\to\infty}\alpha=\lim_{k\to\infty}\frac{\log\frac{1}{2}}{\log(\frac{1}{2}(1-\frac{1}{k}))}=1. 
$$

Let $\tilde{\mathcal{D}}_k$, $k\in\N$, be the collection of $k$-th generation dyadic square  of $[0,1]^2$ into $(2^k)^2$ squares of sidelength $2^{-k}$. It is easy to see that 
$\mathscr{H}^1(\varphi(\partial \tilde{Q}_{k,j}))<\infty$ for all $k$ and $j$ by \eqref{defh}. 
Using H\"older continuity of $h$ we get  
$$
\sum_{k=0}^\infty\sum_{j=1}^{2^{2k}} 2^{-(3-q)k} \mathscr{H}^1(\varphi(\partial \tilde{Q}_{k,j}))^q
\leq  C\sum_{k=0}^\infty 2^{2k} 2^{-(3-q)k} [2^{-\alpha k}]^q. 
$$
This sum is finite for $q(1-\alpha)<1$ and we can choose $k$ large enough so that this condition holds, i.e.~\eqref{eq:discretehomeoextension} holds. By Theorem \ref{thm:mainsum} we obtain that we can extend this boundary homeomorphism as a $W^{1,q}$ homeomorphism inside. However, the mapping $\varphi$ does not belong to $W_{\loc}^{1,1} ([0,1]^2, \R^2)$ as it fails the ACL condition on all vertical segments (it just has bounded variation). 
\end{proclaim}

\section{Decomposition of the domain and target side}\label{sec:standa}
\noindent 
 In this section we start with the standard dyadic decomposition $\tilde{\mathcal{D}}_k$ of the boundary and define a modification of it in order to control the lengths of the image curves of the image grid under the given boundary map $\varphi$. Furthermore, we will define piecewise linear replacements of these image curves. These divisions on the domain and target side will be used in later sections to assist in defining the extension map we use to prove our main result, Theorem \ref{thm:mainsum}. We also show in this section that Theorem \ref{thm:mainsobo} follows from Theorem \ref{thm:mainsum}.

\begin{lemma}\label{lem:refinesquares} Let $\tilde{\mathcal{D}}_k = \{\tilde{Q}_{k,j} : k \in \mathbb{N}, j = 1 \ldots 2^{2k}\}$ be the dyadic decomposition of the unit square $Q_0 = [0,1]^2$ into closed squares of side length $2^{-k}$ for each fixed $k$. Let $p > 1$ and $\varphi : \overline{Q_0} \to \overline{Q_0}$ be a homeomorphism in the space $\varphi \in W^{1,p}(3Q_0,\er^2)$. Then there exists a set of closed quadrilaterals $\mathcal{D}_k = \{Q_{k,j} : k \in \N, j = 1 \ldots 2^{2k}\}$ such that
\begin{enumerate}
\item 
For each point $\tilde{v} \in Q_0$ which is a vertex of a dyadic square of side length $2^{-k}$ in $\tilde{\mathcal{D}}_k$, there exists exactly one corresponding point $v \in Q_0$ which is a vertex of a quadrilateral from $\mathcal{D}_k$. The vertices $v$ of a quadrilateral $Q_{k,j}$ in $\mathcal{D}_k$ are exactly the points which correspond to the vertices $\tilde{v}$ of the dyadic square $\tilde{Q}_{k,j}$. Moreover, for the coordinates of these points $v=[v_1,v_2]$ and $\tilde{v}=[\tilde{v}_1,\tilde{v}_2]$ we have (see Figure \ref{newgrid})
\eqn{choiceofvertices}
$$
v_1-\tilde{v}_1 \in\Bigl[\frac{2^{-k}}{10}-\frac{2^{-k}}{40},\frac{2^{-k}}{10}\Bigr]\text{ and }v_2-\tilde{v}_2 \in\Bigl[\frac{2^{-k}}{10}-\frac{2^{-k}}{40},\frac{2^{-k}}{10}\Bigr]
$$ 
for all pairs of corresponding vertices.
\item The quadrilaterals $Q_{k,j}$ for each fixed level $k$ are thus mutually disjoint apart from their boundaries.
\item If we inherit the parent-child relation between dyadic squares from $\tilde{\mathcal{D}}$ to $\mathcal{D}$, then the following holds. The children $Q_1, \ldots Q_4 \in \mathcal{D}_{k+1}$ of a given square $Q \in \mathcal{D}_k$ (i.e. $Q=Q_1\cup Q_2\cup Q_3\cup Q_4$) need not be contained in $Q$ nor does their union need to cover $Q$. However, for $\hat{Q}=\cup_{i=1}^4 Q_i$ the boundaries $\partial Q$ and $\partial \hat{Q}$ always intersect exactly at two points.
\item For each $k,j$ we have the inequality
\eqn{key}
$$
2^{-k} \int_{\partial Q_{k,j}} |D\varphi(t)|^p dt \leq C \int_{2Q_{k,j}} |D\varphi(z)|^p dz.
$$
\end{enumerate}
\end{lemma}
\begin{figure}[H]
\phantom{a}
\vskip 170pt
{\begin{picture}(0.0,0.0) 
     \put(-140.2,0.2){\includegraphics[width=0.80\textwidth]{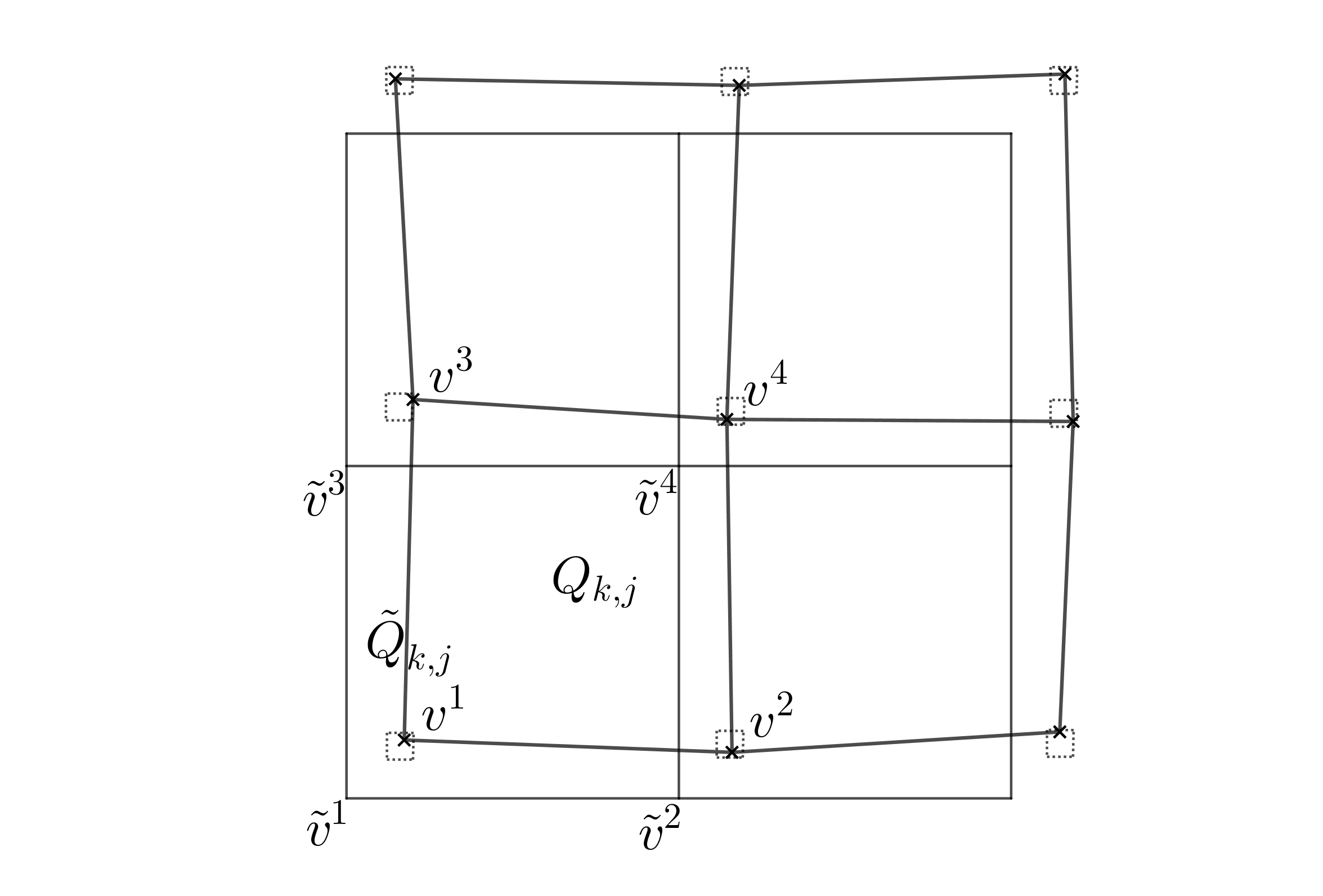}}
  \end{picture}
  }
\vskip -20pt	
\caption{Given a dyadic cube $\tilde{Q}_{k,j}$ with vertices $\tilde{v}^1, \tilde{v}^2, \tilde{v}^3, \tilde{v}^4$ we construct a quadrilateral $Q_{k,j}$ with vertices $v^1, v^2, v^3, v^4$. Each $v^i$ is close to $\tilde{v}^i$, it is slightly shifted to the top and to the right from $\tilde{v}^i$.}\label{newgrid}
\end{figure}
\begin{proof}
{\bf (1) and (4)}: Let us first explain that it is possible to choose the grid so that (1) is satisfied and we have the key inequality \eqref{key}. 

This follows essentially from \cite[Section 4.2]{HP} and therefore we only explain how to apply this approach here: All of our cubes in the $r=2^{-k}$ grid are of type A since we can freely move points outside of $Q_0$. 
We would like to apply analogy of \cite[Lemma 4.9]{HP}  for $M=0$ and $\varepsilon=\frac{1}{10}$. The only difference is that in \cite[Lemma 4.9]{HP} they choose
$$
[v_1,v_2]\in I_{\varepsilon}=\bigl\{[\tilde{v}_1+t,\tilde{v}_2+t]: |t|\leq \varepsilon 2^{-k}\bigr\}
$$ 
but we would like to make this choice in the subset of $I_{\varepsilon}$ (of length $1/8$ times the original length)
$$
[v_1,v_2]\in I=\bigl\{[\tilde{v}_1+t,\tilde{v}_2+t]: 
t\in[\tfrac{1}{10}2^{-k}-\tfrac{1}{40}2^{-k},\tfrac{1}{10}2^{-k}]\bigr\}. 
$$
This does not change anything substantial in the proof there, it only affects some multiplicative constants - use $8^2\frac{25}{\varepsilon r}$ instead of $\frac{25}{\varepsilon r}$ in the definition of $\Gamma(A,B,M)$ and then the proof carries through with obvious minor modifications. 
Then we can finish this step by applying analogy of \cite[Lemma 4.13 and Lemma 4.16]{HP} (again with slightly increased multiplicative constant) to get our \eqref{key}. 

{\bf (2)}: This is easy to see from the definition of vertices of $\Q_{k,j}$ in step (1) (see Fig \ref{newgrid}).

\begin{figure}
\phantom{a}
\vskip 170pt
{\begin{picture}(0.0,0.0) 
     \put(-150.2,0.2){\includegraphics[width=0.80\textwidth]{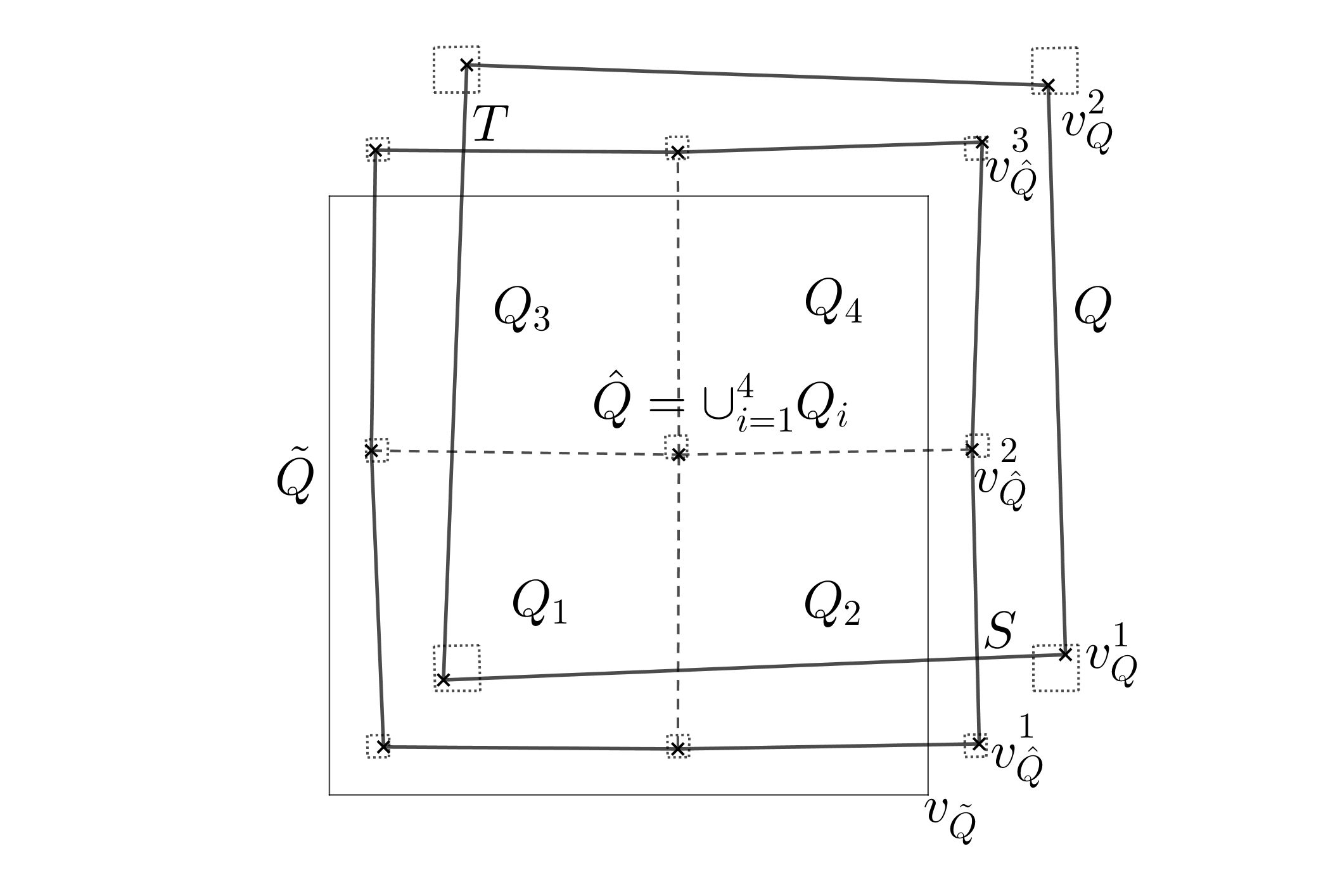}}
  \end{picture}
  }
\vskip -20pt	
\caption{Boundaries of $Q$ and $\hat{Q}=\cup_{i=1}^4 Q_i$ intersect at two points $S$ and $T$.}\label{intersectiongrid}
\end{figure}

{\bf (3)}: Let $Q$ and $\hat{Q}=\bigcup_{i=1}^4 Q_i$ be as in the statement (see Fig \ref{intersectiongrid}). 

Let us denote (as in figure) $v_{\tilde{Q}}$ the vertex of $\tilde{Q}$, $v_Q^1$ and $v_Q^2$ vertices of $Q$ and $v_{\hat{Q}}^1$, $v_{\hat{Q}}^2$, $v_{\hat{Q}}^3$ vertices of $\hat{Q}$ (in fact the corresponding side of $\hat{Q}$ is given by two segments 
$v_{\hat{Q}}^1 v_{\hat{Q}}^2$ and $v_{\hat{Q}}^2v_{\hat{Q}}^3$). From \eqref{choiceofvertices} we obtain for the x-coordinates of these points that
$$
(v_Q^1)_1-(v_{\tilde{Q}})_1, (v_Q^2)_1-(v_{\tilde{Q}})_1\in \Bigl[\frac{2^{-k}}{10}-\frac{2^{-k}}{40},\frac{2^{-k}}{10}\Bigr]
$$
and similarly from \eqref{choiceofvertices} for the choice of $\mathcal{D}_{k+1}$
$$
(v_{\hat{Q}}^1)_1-(v_{\tilde{Q}})_1, (v_{\hat{Q}}^2)_1-(v_{\tilde{Q}})_1, 
(v_{\hat{Q}}^3)_1-(v_{\tilde{Q}})_1\in 
\Bigl[\frac{2^{-(k+1)}}{10}-\frac{2^{-(k+1)}}{40},\frac{2^{-(k+1)}}{10}\Bigr]. 
$$
It follows that the distance of this side of $Q$ (=segment $v_Q^1 v_Q^2$) and 
this side of $\hat{Q}$ (=union of segments 
$v_{\hat{Q}}^1 v_{\hat{Q}}^2$ and $v_{\hat{Q}}^2v_{\hat{Q}}^3$) is at least 
$\frac{2^{-k}}{10}-\frac{2^{-k}}{40}-\frac{2^{-(k+1)}}{10}=\frac{2^{-k}}{40}$ and thus these two sides do not intersect. By a similar reasoning on other sides we obtain that $\partial Q$ and $\partial \hat{Q}$ intersect at exactly two points $S$ and $T$ as in Figure \ref{intersectiongrid}.

Let us also note that the distance of $S$ and $v_Q^1$ (and similarly distance of $S$ and $v_{\hat{Q}_1}$) is at least $\frac{2^{-k}}{40}$ and thus these intersection points are not too close to the vertices of $\partial Q$ and $\partial \hat{Q}$.
\end{proof}

\begin{definition}\label{def:grids} Note that conditions (1)-(3) above do not involve the boundary map $\varphi$. Hence we may define that any set $\mathcal{D}_k$ of quadrilaterals $Q_{k,j}$ satisfying the conditions (1)-(3) is called a \emph{good modification} of the standard dyadic decomposition of $Q_0$. 
\end{definition}




\begin{proof}[Proof of Theorem \ref{thm:mainsobo}]
Let us know that the statement is obvious if $p\geq q$ as we can use the trivial radial extension. In the following we thus assume that $p<q$. 

Given a homeomorphism $\varphi\in W^{1,p}_{\loc}(\er^2,\er^2)$ we were able to find in Lemma \ref{lem:refinesquares} a good modification $\mathcal{D}_k$ of the dyadic grid so that \eqref{key} holds. We could start with a homeomorphism $\varphi\in W^{1,p}(\mathbb{S},\mathbb{S})$ and some analogy of dyadic grid on $\mathbb{S}$. 
Analogously to the proof of Lemma \ref{lem:refinesquares} we could find a good modification $\mathcal{D}_k$ of this grid on $\mathbb{S}$ so that analogy of \eqref{key} holds for $\varphi$. In fact the whole statement could be also obtained locally using bilipschitz change of variables. 
Note that in our dyadic grid $\mathcal{D}_k$ we have $n_k\approx 2^{2k}$ bi-Lipschitz squares of diameter $\approx 2^{-k}$ and $\mathcal{H}^1(\partial Q_{k,j})\approx 2^{-k}$.


In view of Theorem \ref{thm:mainsum} it is now enough to show finiteness of 
\eqref{eq:discretehomeoextension}. 
Using H\"older's inequality, \eqref{key}, $q/p\geq 1$  and $p>\frac{2}{3}q$ we obtain 
$$
\begin{aligned}
\sum_{k=1}^\infty\sum_{j=1}^{n_k} 2^{-(3-q)k} &\mathscr{H}^1(\varphi(\partial Q_{k,j}))^q \\
&\leq\sum_{k=1}^\infty\sum_{j=1}^{n_k}2^{-(3-q)k}\Bigl(\int_{\partial Q_{k,j}}|D\varphi|\Bigr)^q\\
&\leq \sum_{k=1}^\infty\sum_{j=1}^{n_k}2^{-(3-q)k}
\Bigl(\Bigl(\int_{\partial Q_{k,j}}|D\varphi|^p\Bigr)^{\frac{1}{p}}(2^{-k})^{1-\frac{1}{p}}\Bigr)^q\\
&\leq C\sum_{k=1}^\infty 2^{-(3-q)k} 2^{-k(q-\frac{q}{p})}\sum_{j=1}^{n_k}\Bigl(\Bigl(2^k\int_{2  Q_{k,j}}|D\varphi|^p\Bigr)^{\frac{1}{p}}\Bigr)^q\\
&\leq C\sum_{k=1}^\infty 2^{-k(3-\frac{q}{p})}2^{k\frac{q}{p}}\sum_{j=1}^{n_k}\int_{2 Q_{k,j}}|D\varphi|^p\\
&\leq C\sum_{k=1}^\infty 2^{-k(3-2\frac{q}{p})}<\infty.\\
\end{aligned}
$$
\end{proof}

The aim of the next lemma is to consider the modified dyadic grid given by Lemma \ref{lem:refinesquares}. For each level $k$, we then look at the image of the grid of level $k$ under $\varphi$ (specifically the set $\varphi(\cup_j \partial Q_{k,j})$). The aim is to modify this "image grid" so that instead of general Jordan curves it consists of curves which are piecewise linear. It is necessary to preserve both the topology of the image grid and the lengths of the image curves. This piecewise linear approximation will simplify future computations. 

\begin{lemma}\label{lem:piecewiselinear} Let $p \geq 1$ and $\varphi : \overline{Q_0} \to \overline{Q_0}$ be a homeomorphism in the space $\varphi \in W^{1,p}(Q_0,\er^2)$. Let $\mathcal{D}_k$ be the set of modified dyadic quadrilaterals given by Lemma \ref{lem:refinesquares}. In particular, the Jordan curves $\varphi(\partial Q_{k,j})$ for each $Q_{k,j} \in \mathcal{D}_k$ each have finite length. Then for each quadrilateral $Q_{k,j}$ there exists a corresponding closed Jordan curve $\Gamma_{k,j} \subset \overline{Q_0}$ on the image side such that.
\begin{enumerate}
\item Each of the curves $\Gamma_{k,j}$ is piecewise linear.
\item Each point on the curve $\Gamma_{k,j}$ is of distance at most $2^{-k}$ from the set $\varphi(\partial Q_{k,j})$.
\item The inequality $\mathscr{H}^1(\Gamma_{k,j}) \leq \mathscr{H}^1(\varphi(\partial Q_{k,j}))$ holds.
\item $\Gamma_{k,j}$ passes through the four points $\varphi(v)$, where $v$ ranges over the four vertices of the quadrilateral $Q_{k,j}$. These four points are called the vertices of $\Gamma_{k,j}$.
\item If two quadrilaterals $Q_{k,j},Q_{k,j'} \in \mathcal{D}_k$ share a common side with endpoints $v_1,v_2$, then the subarcs of their corresponding image curves $\Gamma_{k,j}, \Gamma_{k,j'}$ with endpoints at the common vertices $\varphi(v_1)$ and $\varphi(v_2)$ are the same.
\item Apart from the cases where two curves $\Gamma_{k,j}, \Gamma_{k,j'}$ at the same level $k$ share either a single vertex or a single subarc between two vertices as before, these Jordan curves are mutually disjoint (for each fixed level $k$). 
\item For every $Q_{k,j}\in \mathcal{D}_k$ and $Q_{k+1,j'}\in \mathcal{D}_k$ (see Fig. \ref{intersectiongrid}) we know that 
$$
\Gamma_{k,j}\cap \Gamma_{k+1,j'}=\varphi(\partial Q_{k,j})\cap \varphi(\partial Q_{k+1,j'}) .
$$
That is each $\Gamma_{k,j}$ passes not only through its vertices but also through its intersection with grids of step $k+1$ and $k-1$, i.e. images of boundaries of $\mathcal{D}_{k+1}$ and $\mathcal{D}_{k-1}$. 
\end{enumerate}
\end{lemma}
\begin{proof}
In this proof we use ideas of \cite{DP} and \cite{HP} where a similar piecewise linear approximation of curves was used. We first explain how to do this for a single level $\mathcal{D}_k$ and then we explain that we can even manage that (7) is satisfied. 

\begin{figure}
\phantom{a}
\vskip 160pt
{\begin{picture}(0.0,0.0) 
     \put(-150.2,0.2){\includegraphics[width=0.90\textwidth]{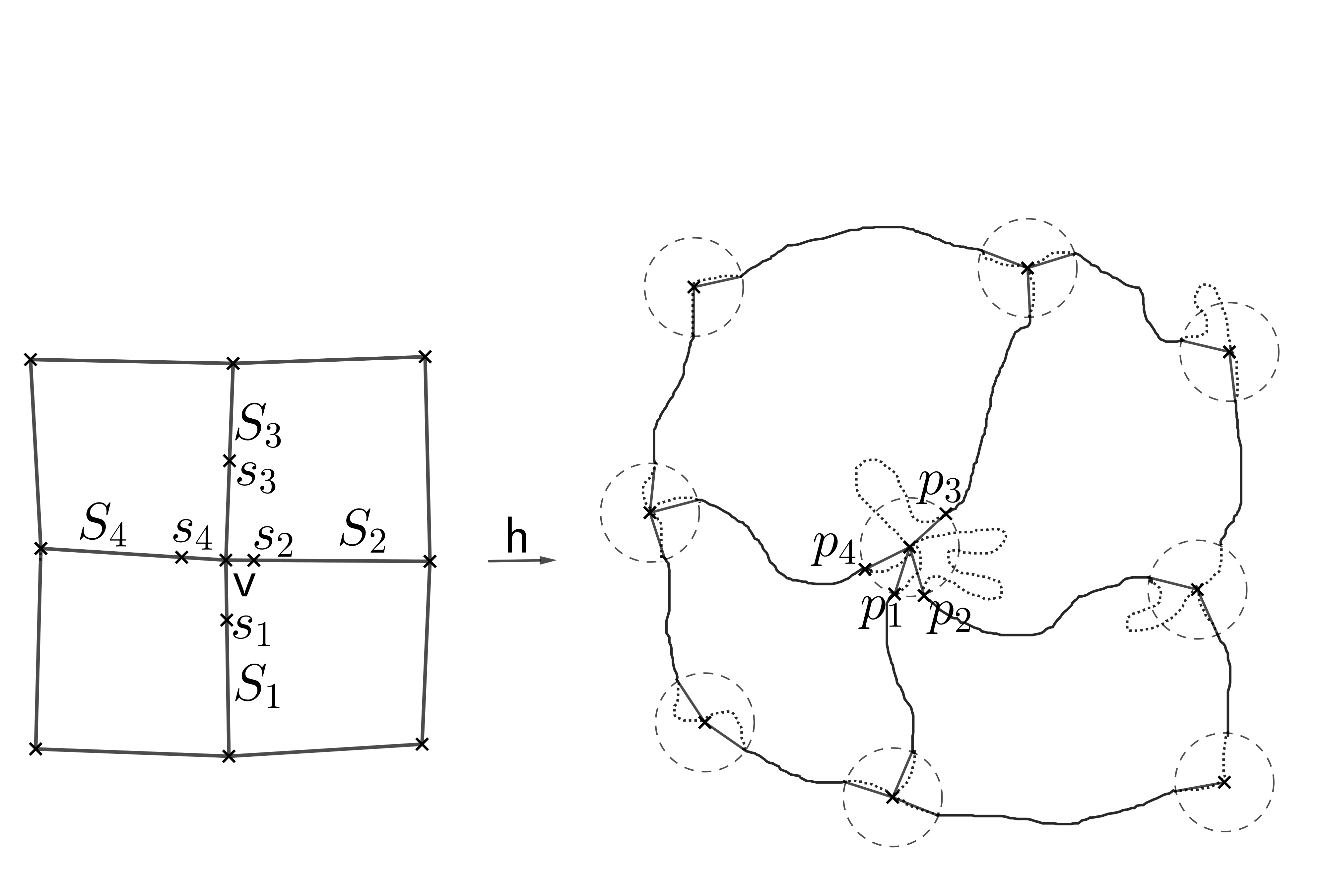}}
  \end{picture}
  }
\vskip -20pt	
\caption{We replace original curve near vertices (see dotted curves) by segments near vertices.}\label{linearvertex}
\end{figure}

{\underline{Step 1. Linearization near vertices:}} For each vertex $v$, $v$ is a vertex of some $Q_{k,j}$, we choose a ball $B(\varphi(v),r)$. We choose $r>0$ small enough so that balls $B(\varphi(v),2r)$ are pairwise disjoint and (using uniform continuity of $\varphi^{-1}$ and $\varphi$) so that 
\eqn{notfar}
$$
\text{ for every }x\in B\bigl(v,\diam(\varphi^{-1}(B(\varphi(v),r))\bigr)\text{ we have }|\varphi(x)-\varphi(v)|<2^{-k}. 
$$

For each vertex $v$ we have four sides $S_1,S_2,S_3$ and $S_4$ of some $Q_{k,j}$ that have $v$ as their endpoint (see Fig. \ref{linearvertex}). On each of these sides we choose points $s_i\in S_i$ so that $p_i=\varphi(s_i)\in \partial B(\varphi(v),r)$ and so that $s_i$ is furthest away from $v$ with this property (e.g. on $S_3$ in Fig. \ref{linearvertex} we have three points whose image intersects $\partial B(\varphi(v),r)$). Now we replace $\varphi$ on each segment $[s_i,v]$ by a segment $[p_i, \varphi(v)]$ and we leave $\varphi$ the same outside of these four segments (see Fig. \ref{linearvertex}). In this way we replace $\varphi(\partial Q_{k,j})$ by a curve $\tilde{\Gamma}_{k,j}$ which is piecewise linear close to the vertices. 

It is easy to see that this new curve $\tilde{\Gamma}_{k,j}$ satisfies analogy of (2) by \eqref{notfar} and it is not difficult to see that these new curves are one-to-one (see Fig. \ref{linearvertex}), i.e. they intersect only at original vertices $v$. These new curves have also length  shorter or equal to the original $\mathscr{H}^1(\varphi(\partial Q_{k,j}))$.   

\begin{figure}
\phantom{a}
\vskip 180pt
{\begin{picture}(0.0,0.0) 
     \put(-220.2,0.2){\includegraphics[width=1.10\textwidth]{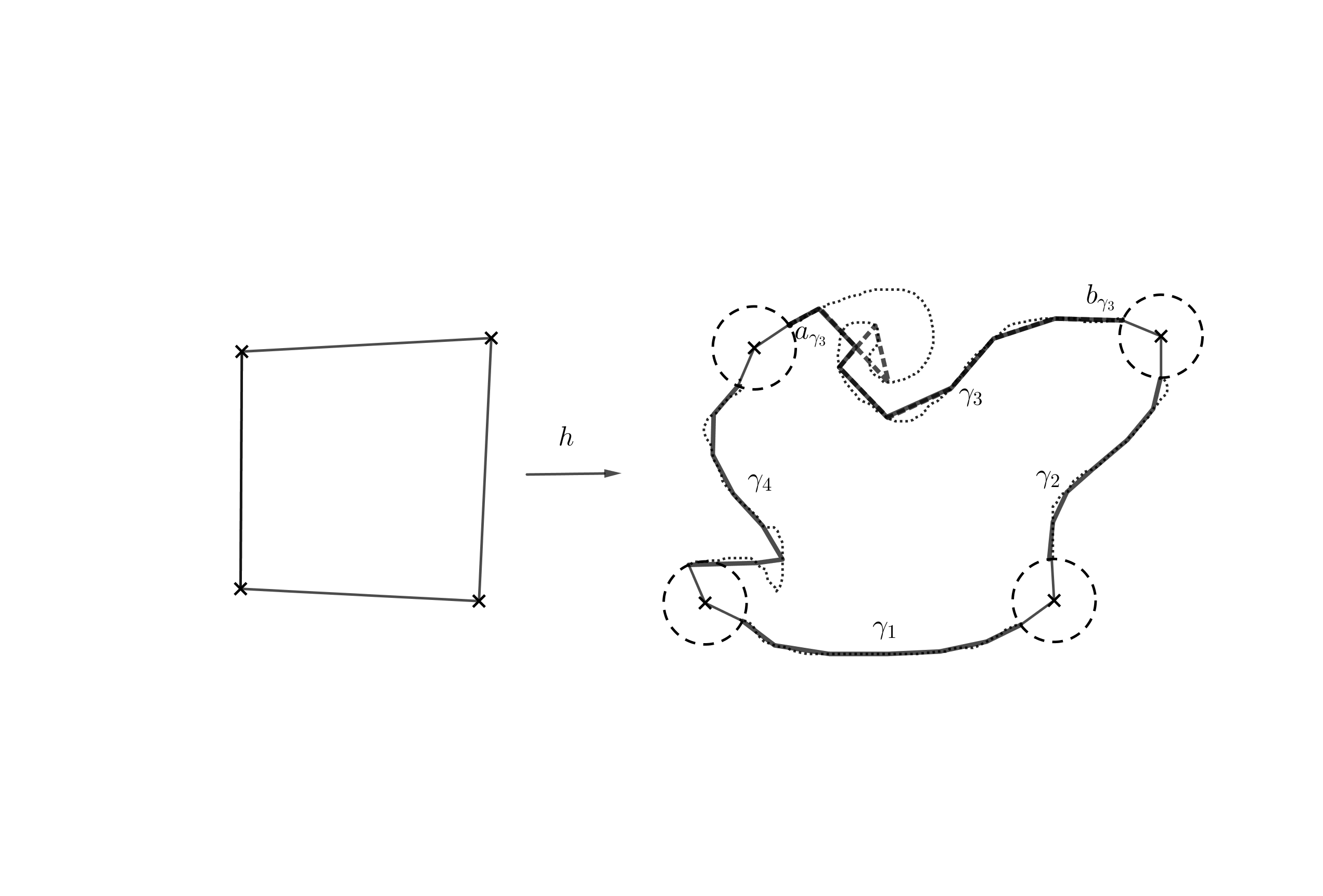}}
  \end{picture}
  }
\vskip -60pt	
\caption{We replace curves $\gamma_m$ on the sides (see dotted curves) by piecewise linear curves. We may need to choose a one-to-one shortening of these replacements, i.e we ignore some dashed part of the replacement of $\gamma_3$.}\label{linearsides}
\end{figure}

{\underline{Step 2. Linearization of sides:}} Now we need to change $\tilde{\Gamma}_{k,j}$ so it is piecewise linear not only close to the vertices. We call $\gamma_{k,m}$ the parts of $\tilde{\Gamma}_{k,j}$ where our curve is not piecewise linear yet, these correspond to image by $\varphi$ of sides of $Q_{k,j}$ (minus segments $[s_i,v]$ near vertices). These 
$\gamma_{k,m}$ are pairwise disjoint and we can choose $0<\delta<2^{-k}$ so that $\gamma_{k,m}+B(0,2\delta)$ are pairwise disjoint. We choose enough division points in $\gamma_{k,m}$ and we connect them by segments (see Fig. \ref{linearsides}) so that the union of these segments approximates the original curve. We definitely include two endpoints $a_{\gamma_{k,m}}$ and $b_{\gamma_{k,m}}$ in these division points and we assume that we have so many division points so that the union of these segments lies inside $\gamma_{k,m}+B(0,\delta)$. It follows that these segments for different $\gamma_{k,m}$ do not intersect. 

However, it may happen that they intersect (see $\gamma_3$ in Fig \ref{linearsides}) for a given $\gamma_{k,m}$. In this case we simply choose a shortest path in the union of these segments between the endpoints $a_{\gamma_{k,m}}$ and $b_{\gamma_{k,m}}$ and we replace the union of these segment by this shortest path (see the right side of Fig \ref{linearsides}). It is not difficult to see that by this replacement we get a one-to-one piecewise linear curve that replaces $\gamma_{k,m}$.  
Now we call $\Gamma_{k,j}$ the corresponding piecewise linear approximation of $\tilde{\Gamma}_{k,j}$. 
It is easy to see that we have $(1)$, $(2)$ (using $\delta<2^{-k}$), $(3)$, $(4)$, $(5)$ and $(6)$ for our $\Gamma_{k,j}$.

{\underline{Step 3. Intersection of $\Gamma_{k,j}$ and $\Gamma_{k+1,j'}$}:} 
We need to do linearization of the grid that not only preserves the vertices but also preserves the intersection of neighboring grids. We define the grids  
$$
\mathcal{G}_0=\emptyset \text{ and }\mathcal{G}_k=\bigcup_j \varphi(\partial Q_{k,j})
$$
and the set of vertices for $k\in\en$ as 
$$
\mathcal{V}_k=\bigl\{\varphi(v):\ v\text{ is a vertex of some }Q_{k,j}\bigr\}\cup
\bigl(\mathcal{G}_k\cap \mathcal{G}_{k+1}\bigr)\cup \bigl(\mathcal{G}_k\cap \mathcal{G}_{k-1}\bigr). 
$$
Analogously to the reasoning in the proof of Lemma \ref{lem:refinesquares} $(3)$ we obtain that $\varphi^{-1}(\mathcal{G}_k\cap \mathcal{G}_{k+1})$ is finite (see Fig. \ref{intersectiongrid}) and thus it is not difficult to see that $\mathcal{V}_k$ is finite. Moreover, it is possible to show analogously to the proof Lemma \ref{lem:refinesquares} $(3)$ that there is $C>0$ with (see Fig. \ref{intersectiongrid}) 
\eqn{faraway}
$$
|\varphi^{-1}(y)-\varphi^{-1}(z)|\geq C 2^{-k}\text{ for every distinct }y,z\in \mathcal{V}_k.
$$

We choose $r_k>0$ so that $B(v,2 r_k)$, $v\in \mathcal{V}_k$, are pairwise disjoint and so that 
$$
\text{ for every }x\in B\bigl(v,\diam(\varphi^{-1}(B(\varphi(v),r_k))\bigr)\text{ we have }|\varphi(x)-\varphi(v)|<2^{-k}. 
$$
We further assume that $r_{k+1}<r_k$ and we choose balls around $\mathcal{V}_k$ as 
$$
B(v,r_{k+1})\text{ for }v\in \mathcal{G}_k\cap \mathcal{G}_{k+1}\text{ and }
B(v,r_{k})\text{ for other }v\in \mathcal{V}_k. 
$$
In each such a ball we do a linearization as in Step 1. Note that this works fine as 
for $v\in \mathcal{G}_k\cap \mathcal{G}_{k+1}$ we have $B(v,r_{k+1})$ both for linearization of $\mathcal{G}_k$ and for linearization of $\mathcal{G}_{k+1}$ near this vertex so that the corresponding $\tilde{\Gamma}_{k,j}$ and $\tilde{\Gamma}_{k+1,j'}$ intersect only at vertices in $\mathcal{G}_k\cap \mathcal{G}_{k+1}$. 

As in Step 2. we call $\gamma_{k,m}$ the parts of $\tilde{\Gamma}_{k,j}$ where our curve is not piecewise linear yet, these correspond to image by $\varphi$ of sides of $Q_{k,j}$ (minus segments near all vertices of $\mathcal{V}_k$ where the curve is already linear). We choose $\delta_k<2^{-k}$ small enough so that not only 
$$
\gamma_{k,m}+B(0,2\delta_k)\text{ are piecewise disjoint}
$$
but also 
$$
\gamma_{k,m}+B(0,2\delta_k), \gamma_{k-1,m'}+B(0,2\delta_k)\text{ and }\gamma_{k+1,m''}+B(0,2\delta_k)\text{ do not intersect}. 
$$
We assume that $\delta_{k+1}<\delta_k$ and as in Step 2. we linearize $\gamma_{k,m}$ so that the corresponding piecewise linear curve is one-to-one and stays inside $\gamma_{k,m}+B(0,\delta_k)$. In this way we obtain $\Gamma_{k,j}$ as the linearization of $\varphi(Q_{k,j})$. 

Again it is easy to see that we have $(1)$, $(2)$ (using $\delta<2^{-k}$), $(3)$, $(4)$, $(5)$ and $(6)$ for our $\Gamma_{k,j}$. Moreover, it is not difficult to check that $(7)$ also holds in this situation. 
\end{proof}

{\bf Parametrization of $\Gamma_{k,j}$:} We have constructed a piecewise linear curve $\Gamma_{k,j}$ that approximated $\varphi(Q_{k,j})$ and keeps images of vertices in $\mathcal{V}_k$ fixed. We know that there are four 
$y\in\mathcal{V}_k$ such that $y=\varphi(v)$ for some vertex of $Q_{k,j}$. 
Further there are at most $8$ points in 
$$
\mathcal{G}_{k+1}\cap \varphi(Q_{k,j})=\mathcal{G}_{k+1}\cap \Gamma_{k,j}
$$
that is on image of each side of $Q_{k,j}$ there are at most two (see Fig \ref{intersectiongrid} and proof of Lemma \ref{lem:refinesquares} $(3)$). Further we have at most two points in $\mathcal{G}_{k-1}\cap \varphi(Q_{k,j})$, see Lemma \ref{lem:refinesquares} $(3)$. As we have already noted in \eqref{faraway} the distance of preimages of these points is comparable to sidelength of $Q_{k,j}$, i.e. $2^{-k}$. 

Now we divide $\Gamma_{k,j}$ into at most $4+8+2=14$ pieces $P_i$ by points in $\mathcal{V}_k$. For points 
$x\in \varphi^{-1}(\mathcal{V}_k\cap Q_{k,j})$  we define $p(x)=\varphi(x)$ so that our parametrization $p$ has the same value as original mapping $\varphi$ on these "vertices" and intersection points. We parametrize pieces $P_i$ by constant speed parametrization $p$ there, i.e. on each of those pieces it has constant speed which might be different for each piece. Since the length of these pieces is bounded by $\mathscr{H}^1(\varphi(Q_{k,j}))$ we obtain using \eqref{faraway} that 
$$
|Dp|\leq C\frac{\mathscr{H}^1(\varphi(Q_{k,j}))}{2^{-k}}\text{ on the whole }Q_{k,j}. 
$$


\section{The 2D extension}\label{sec:2d}

Let $S$ be the square with vertices at $\{(1,0),(0,1),(-1,0),(0,-1)\}$ and $\Y$ be a Jordan domain with piecewise linear boundary. Suppose that a boundary homeomorphism $\varphi : \partial S \to \partial \Y$ is given. We now describe a way to extend $\varphi$ as a homeomorphism of $\bar{S}$ to $\bar{\Y}$ with Lipschitz-continuity controlled by the boundary map.

First, we describe an extension $H_\varphi$ of $\varphi$ which is a \emph{monotone map} from $\bar{S}$ to $\bar{\Y}$, meaning it is continuous and the preimage of every point is connected. The final homeomorphic extension will be obtained via an arbitrarily small modification of $H_\varphi$ as we are able to describe the points where it fails to be injective and fix them accordingly. However, this modification will be done only later in Section \ref{sec:injective}.

The extension $H_\varphi$ will also be called the \emph{shortest curve extension} of $\varphi$. To define $H_\varphi$, we let $l_s$ denote the horizontal line segment which is obtained as the intersection between the line $\{(x,y) : y = s\}$ and $S$. This segment $l_s$ has two endpoints $a_s$ and $b_s$ (from left to right) on $\partial S$. We let $A_s = \varphi(a_s)$, $B_s = \varphi(b_s)$, and define $L_s$ as the shortest curve in $\bar{S}$ which connects $A_s$ to $B_s$.

The map $H_\varphi$ is now given by defining it to map each horizontal segment $l_s$ to the corresponding shortest curve $L_s$ via constant speed parametrization. It is simple to verify that this mapping is continuous.

\begin{lemma}\label{lem:shortestcurve} If $\varphi: \partial S \to \partial \Y$ is Lipschitz with constant $L$, then the shortest curve extension $H_\varphi$ is also Lipschitz with constant at most $CL$ for a uniform constant $C$.
\end{lemma}
\begin{proof}

\emph{Case 1.} Lipschitz continuity in the horizontal direction.\\\\
We show that $H_\varphi$ satisfies the required Lipschitz-continuity on each of the horizontal segments $l_s$. For this, note that the constant speed parametrization on each of these segments implies that we only need to show that $|L_s| \leq 2L |l_s|$, where $|\cdot|$ denotes the one-dimensional Hausdorff measure. The endpoints of $l_s$ separate $\partial S$ into two connected components, the shorter of which we may call $\gamma_s$. Since $L_t$ is the shortest curve from $A_s$ to $B_s$, we find that $|\varphi(\gamma_s)| \geq |L_s|$. However, due to the Lipschitz-continuity of $\varphi$ we must have that $|\varphi(\gamma_s)| \leq L |\gamma_s|$. Thus
\[|L_s| \leq |\varphi(\gamma_s)| \leq L |\gamma_s| \leq 2 L |l_s|,\]
where the last inequality is due to the fact that $l_s$ is the hypotenuse of a right-angled triangle with sides given by $\gamma_s$.\\\\
\emph{Case 2.} Lipschitz continuity in the vertical direction.\\\\
Let us fix $s \in (-1,1)$ and pick a point $z \in l_s$. For small $\delta$ we let $z_\delta = z + i\delta$ and our aim is to show that $|H_\varphi(z_\delta) - H_\varphi(z)| \leq C L \delta$. As Lipschitz-continuity is a local property, we may assume that $\delta$ is arbitrarily small. In fact, to simplify calculations we assume that $\delta$ is very small compared to $|l_s|$, which lets us assume that the trapezium bounded by the segments $l_s$ and $l_{s+\delta}$ is actually a rectangle with longer sides of length $|l_s|$ due to the fact that these two shapes are bilipschitz-equivalent with a uniform constant (say $2$) for small enough $\delta$.

Consider the curves $L_s$ and $L_{s+\delta}$. By choosing $\delta$ small enough, we may assume that the endpoints $A_s$ and $A_{s+\delta}$ lie on the same line segment of the piecewise linear boundary $\partial \Y$. The same may be assumed for $B_s$ and $B_{s+\delta}$. Now basic geometry dictates that the curves $L_s$ and $L_{s + \delta}$ must each consist of three parts as follows (for a detailed argument, see \cite{HP}). See also Figure \ref{fig:short4}.

\begin{enumerate}
\item $\alpha_s$ and $\alpha_{s+\delta}$: Curves which start from $A_s$ and $A_{s+\delta}$ and do not intersect except at their common other endpoint. In fact, if $\delta$ is assumed small enough these curves may be assumed to be line segments.
\item A common part of $l_s$ and $L_s$, which is a piecewise linear curve we denote by $\gamma$.
\item $\beta_s$ and $\beta_{s+\delta}$: Analogously to the first part, these can be assumed to be line segments from $B_s$ and $B_{s+\delta}$ respectively which meet at a common point (the other endpoint of $\gamma$).
\end{enumerate}

\begin{figure}[H]\label{fig:short4}
\includegraphics[scale=0.6]{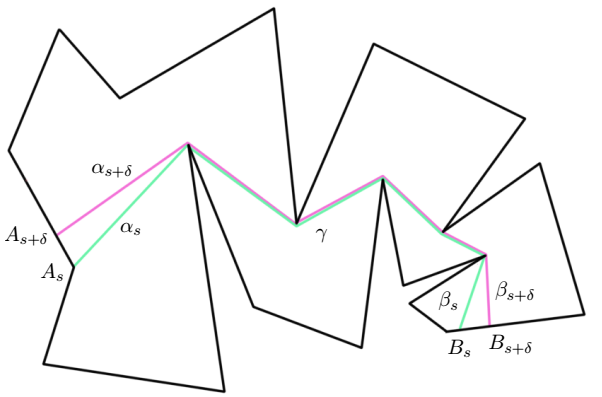}
\caption{The shortest curves $L_s$ and $L_{s+\delta}$, split into three parts.}
\end{figure}

We may assume that $H_\varphi(z)$ lies on either $\alpha_s$ or $\gamma$ as the case where it lies on $\beta_s$ is handled by symmetry. Let $D$ denote the line segment between $a_s$ and $a_{s+t}$. Then since $\varphi$ is $L$-Lipschitz-continuous on $\partial S$, we find that $|D| \leq L \delta$. By the triangle inequality we obtain that $||\alpha_s| - |\alpha_{s+\delta}|| \leq L \delta$ and using the same argument for the $\beta$-curves gives $||L_s| - |L_{s+\delta}|| \leq 2 L \delta$. Let also $d$ denote the distance between $z$ and $a_s$, which is also the distance from $z_\delta$ to $a_{s+\delta}$.

Suppose first that $H_\varphi(z)$ lies on $\gamma$. The length of the part of $L_s$ between $A_s$ and $H_\varphi(z)$ may now be calculated in two ways. The constant speed parametrization tells us that it is equal to $|L_s| d/|l_s|$. On the other hand, it is also equal to $|\alpha_s| + |\gamma'|$, where $\gamma'$ denotes the part of $\gamma$ between $\alpha_s$ and $H_\varphi(z)$. Thus
\[|\alpha_s| + |\gamma'| = \frac{|L_s| d}{|l_s|}.\]
If $\Gamma$ denotes the part of $L_{s+\delta}$ between $H_\varphi(z)$ and $H_\varphi(z_\delta)$, then we may calculate the length of the part of $L_{s+\delta}$ between $a_{s+\delta}$ and $H_\varphi(z_\delta)$ in two ways similarly as above to obtain that
\[|\alpha_{s+\delta}| + |\gamma'| \pm |\Gamma| = \frac{|L_{s+\delta}| d}{|l_s|}.\]
The $\pm$ in this equation is there to account for the two cases on which side of $L_{s+\delta}$ the point $H_\varphi(z_\delta)$ lies in comparison to $H_\varphi(z)$. In either case, we find by combining the above two equalities that
\begin{align*}|\Gamma| &\leq ||\alpha_s| - |\alpha_{s+\delta}|| + ||L_s| - |L_{s+\delta}|| \frac{d}{|l_s|}\\
&\leq L \delta + 2 L \delta.
\end{align*}
This shows that $|H_\varphi(z_\delta) - H_\varphi(z)| \leq 3L \delta$.

Suppose then that $H_\varphi(z)$ lies on $\alpha_s$. The length of the part of $\alpha_s$ from $a_s$ to $H_\varphi(z)$ must then be equal to $|L_s| d/|l_s|$ by constant speed parametrization. Let $\omega$ be a point on $\alpha_{s+\delta}$ of distance at most $|D|$ from $H_\varphi(z)$, which is possible to choose due to convexity. Let $\gamma^*$ denote the part of $\alpha_{s+\delta}$ between $a_{s+\delta}$ and $\omega$, and $\Gamma$ the part of $L_{s+\delta}$ between $\omega$ and $H_\varphi(z_\delta)$. By triangle inequality,
\[\left\||\gamma^*| - \frac{|L_s| d}{|l_s|}\right\| \geq  - 2L \delta.\]
Thus we find that
\begin{align*}
|\Gamma| &\leq \left\| \frac{|L_{s+\delta}| d}{|l_s|} - |\gamma*|\right\|
\\&\leq ||L_s| - |L_{s+\delta}|| \frac{d}{|l_s|} + 2L \delta
\\&\leq 4L\delta.
\end{align*}
This shows that $|H_\varphi(z_\delta) - H_\varphi(z)| \leq 4L \delta$ and proves our claim.

Note: We will use the following consequence of this proof repeatedly in multiple other parts of the paper. Given a Jordan domain $\yy$ with a piecewise linear boundary and points $A_1,A_2, B \in \partial \yy$, suppose that the part of $\partial \yy$ between $A_1$ and $A_2$ which does not contain $B$ has length $\delta'$. Then if $\varphi_1, \varphi_2 : [0,1] \to \bar{\yy}$ are the two shortest curves in $\bar{\yy}$ from $B$ to $A_1$ and $A_2$ respectively, parametrized with constant speed, then $|\varphi_1(x) - \varphi_2(x)| \leq C\delta'$ for all $x \in [0,1]$. This claim follows from the above proof, notably the only difference is that we start from the same point $B$ instead of two points $B_s$ and $B_{s+\delta}$ but this case is even simpler.

\end{proof}

\subsection{Lipschitz-continuity in the time variable}

Our next aim is to look at a situation where instead of a single given boundary map $\varphi$, we are given a continuous sequence of boundary homeomorphisms $\varphi_t : \partial S \to \rr^2, t \in [0,1]$ (not necessarily to the same target domain). The aim is to show that if the dependence on $t$ is Lipschitz, meaning that
\begin{equation}\label{eq:tlip}
|\varphi_{t_1}(z) - \varphi_{t_2}(z)| \leq L |t_1-t_2| \qquad \text{ for } z \in \partial S,
\end{equation}
Then the same estimate holds (up to a uniform constant) for the extensions $H_{\varphi_t}$ and points $z \in \bar{S}$ as well. We expect this to be true in the general case, but for our purposes we will only need to prove such a result in a few simple cases which are easier to explain.
\begin{lemma}\label{lem:sameboundary} Suppose that $\Y \subset \C$ is a piecewise linear Jordan domain and $\varphi_t : \partial S \to \partial \Y$ are given boundary homeomorphisms so that \eqref{eq:tlip} is valid. Suppose also that the maps $\varphi_t(z)$ are equal on one half of $\partial S$, say $\varphi_t(z) =  \varphi_0(z)$ for all $z \in \partial S$ with $\Re z \leq 0$. Then $t \mapsto H_{\varphi_t}(z)$ is $CL$-Lipschitz for a uniform constant $C$ and all $z \in \bar{S}$.
\end{lemma}
\begin{proof}
Let $z \in \bar{S}$. We consider the horizontal segment $l$ passing through $z$ and its two endpoints $a$ and $b$. Fixing the point $t_1 \in (0,1)$, by continuity we choose $t_2 \in (0,1)$ close enough to $t_1$ so that $\varphi_{t_1}(b)$ and $\varphi_{t_2}(b)$ lie on the same segment on $\partial \Y$. By our assumptions also $\varphi_{t_1}(a) = \varphi_{t_2}(a)$. For $\varphi_{t_1}$, we let $L^{t_1}$ denote the shortest curve from $\varphi_{t_1}(a)$ to $\varphi_{t_1}(b)$ in $\bar{\Y}$. Similarly $L^{t_2}$ is the shortest curve from $\varphi_{t_1}(a)$ to $\varphi_{t_2}(b)$. Then $H_{\varphi_{t_1}}(z)$ lies on $L^{t_1}$ and $H_{\varphi_{t_2}}(z)$ lies on $L^{t_2}$ and the exact positioning of these points on these curves is again determined by the constant-speed parametrization on the horizontal segment $l$. But this situation is essentially exactly the same as in the second case of the proof of Lemma \ref{lem:shortestcurve} (see note at the end of that proof), and we may apply the same proof to show that
\[|H_{\varphi_{t_1}}(z) - H_{\varphi_{t_2}}(z)| \leq 4L|t_1-t_2|.\]
\end{proof}

We now show that given two Lipschitz boundary maps which are equal on one half of $\partial S$, one is able to construct a homotopy between such maps with comparable Lipschitz constant in both the space and time variable.

\begin{lemma}\label{lem:changeboundary}
Suppose that $\varphi_0, \varphi_1 : \partial S \to \rr^2$ are two embeddings of the square $\partial S$ into $\rr^2$. Let $\yy_0$ and $\yy_1$ be the Jordan domains bounded by the respective image curves $\varphi_0(\partial S)$ and $\varphi_1(\partial S)$. Suppose that $\varphi_0(z) = \varphi_1(z)$ for all $z \in \partial S$ with $\Re z \leq 0$, i.e. on the two leftmost sides of square $S$. Let us call the union of these leftmost sides $s_-$ and the union of the two remaining sides $s_+$. Suppose that the curves $\varphi_0(s_+)$ and $\varphi_1(s_+)$ do not intersect except for their endpoints. Suppose also that both embeddings $\varphi_0$ and $\varphi_1$ are Lipschitz-continuous with constant $L$. Then there exists a homotopy $\varphi_t : \partial S \to \rr^2$, $t \in (0,1)$ of embeddings of $\partial S$ between $\varphi_0$ and $\varphi_1$ such that the maps $H_{\varphi_t} : S \to \rr^2$ are also Lipschitz-continuous in $(z,t)$ with constant $C L$ for an uniform constant $C$. Moreover, $\varphi_t(z) = \varphi_0(z)$ for $z \in s_-$ and $\varphi_t(s_+)$ lies between the curves $\varphi_0(s_+)$ and $\varphi_1(s_+)$ for all $t$. Also, $\varphi_t$ may be chosen so that the curves $\varphi_t(s_+)$ do not intersect each other in $t$ except for the mutual endpoints.
\end{lemma}

\begin{proof} Let $\gamma_0 = \varphi_0(s_+)$ and $\gamma_1 = \varphi_1(s_+)$. We first describe a homotopy $\gamma_t$ between these two curves, which will then be used to construct $\varphi_t$ by setting $\varphi_t(s_+) = \gamma_t$ and fixing a parametrization. On $s_-$ we naturally set $\varphi_t \equiv \varphi_0$. At first this homotopy will be constructed in a way such that the curves $\gamma_t$ may mutually overlap but we will modify them slightly to address this later.

The curve $\gamma_t$ is defined as follows. Let the mutual endpoints of $\gamma_0$ and $\gamma_1$ be $A$ and $B$ and the domain between these curves be denoted by $\hat{\yy}$. Let $\gamma_{1/2}$ be the shortest path from $A$ to $B$ within the closure of $\hat{\yy}$. We now need to only describe how to deform $\gamma_0$ to $\gamma_{1/2}$ as the case from $\gamma_{1/2}$ to $\gamma_1$ will be handled in the same way.

For $t \in [0,1/2]$, note that $2t$ varies from $0$ to $1$. We choose $\gamma_t$ as follows. First, travel along $\gamma_0$ starting from $A$ until we have travelled exactly portion $2t$ of $\gamma_0$. We have arrived at a point of $\gamma_0$ which we shall call $P_t$. For the remainder of the parametrization, we take the shortest curve from $P_t$ to $B$ within the closure of $\hat{\yy}$. This defines $\gamma_t$ up to parametrization, and the exact parametrization of $\gamma_t$ will be defined now.

First we note that we may assume that the map $\varphi_0$ maps $s_+$ to $\gamma_0$ with constant speed. If this was not the case, we may deform the parametrization of $\varphi_0$ into a constant speed one while keeping the same Lipschitz constant simply by making a linear homotopy to the identity map in the parameter space. Then Lemma \ref{lem:sameboundary} shows that the same Lipschitz estimate works in the interior as well, in which case we are reduced to the case of constant speed parametrization of $\gamma_0$.

Now, we define $\varphi_t$ for $t \in [0,1/2]$ by setting $\varphi_t(s) = \varphi_0(s)$ for those $s$ which lie within the portion $2t$ of $s_+$ starting from the preimage of $A$. The constant speed parametrization guarantees that then $\varphi_t(2t) = P_t$, meaning that we have travelled the same portion on the domain side on $s_+$ and the image side on $\gamma_0$. For the remaining portion $1 - 2t$ of $s_+$, we also parametrize $\gamma_t$ by constant speed to the respective image curve which is the shortest curve from $P_{2t}$ to $B$ (although the constant may differ from the previous one).

Clearly $\varphi_t(s)$ has uniform Lipschitz-continuity in $s$, so we investigate the estimates in $t$. Fix $s$ and let $0 < t_1 < t_2 < 1/2$. We abuse notation and identify $s_+$ with the interval $[0,1]$ for the moment. If $s \leq 2t_1$, then $\varphi_{t_1}(s) = \varphi_{t_2}(s)$ and there is nothing to consider. The main case is when $s \geq 2t_2$, which we now consider.

This case essentially reduces to the proof of Lemma \ref{lem:shortestcurve} again. We assume that $t_1,t_2$ are close enough so that $P_{t_1}$ and $P_{t_2}$ are on the same segment of the piecewise linear curve $\gamma_0$. Now we are dealing with two curves which are the shortest curves in the closure of $\hat{\yy}$ from $B$ to $P_{t_1}$ and $P_{t_2}$, let us label these $\beta_1$ and $\beta_2$ respectively and suppose that they are parametrized with constant speed from $[2t_1,1]$ and $[2t_2,1]$. The distance between $P_{t_1}$ and $P_{t_2}$ is equal to $|\varphi_0(2t_1) - \varphi_0(2t_2)| \leq L|t_1-t_2|$. We wish to show that $|\beta_1(s) - \beta_2(s)| \leq CL|t_1-t_2|$. The only difference now compared to Lemma \ref{lem:sameboundary} is that there is a slight difference in parametrization. Indeed, the shortest curves $\beta_1$ and $\beta_2$ which we consider here have a different domain of definition. However, we may let $\beta_2^*$ be a curve which has the same image curve as $\beta_2$ and is parametrized with constant speed over the interval $[2t_1,1]$ instead. In this case the same arguments from Lemma \ref{lem:sameboundary} show that
\begin{equation}\label{eq:beta1}|\beta_1(s) - \beta_2^*(s)| \leq CL|t_1-t_2|.\end{equation}
Now, let $s' \in [2t_2,1]$ be such that $\beta_2^*(s) = \beta_2(s')$. Comparing lengths, we must have due to constant speed parametrization that
\[\frac{s'-2t_2}{1-2t_2} = \frac{s-2t_1}{1-2t_1} \Rightarrow s' = 2t_2 + \frac{1-2t_2}{1-2t_1} (s-2t_1).\]
Moreover, since $\beta_2$ is shorter than the part of the curve $\gamma_0$ from $P_{t_2}$ to $B$, due to the Lipschitz estimate and constant speed parametrization we must have that $|\beta_2|/(1-2t_2) \leq L$. We may now estimate that
\begin{align*}
|\beta_2(s) - \beta_2^*(s)| &=  |\beta_2(s) - \beta_2(s')|
\\&= \frac{|\beta_2|}{1-2t_2} |s - s'|
\\&\leq L |s - s'|
\\&= L \left|s - 2t_2 - \frac{1-2t_2}{1-2t_1} (s-2t_1)\right|
\\&= L \left| \frac{(s-2t_2)(1-2t_1) - (s-2t_1)(1-2t_2}{1-2t_1}\right|
\\&= L \left| \frac{s-2t_1}{1-2t_1} (2t_2 - 2t_1) + 2t_1 - 2t_2 \right|
\\&\leq 4L|t_1-t_2|.
\end{align*}
Combining \eqref{eq:beta1} with the above now gives that $|\beta_1(s) - \beta_2(s)| \leq (C+4)L|t_1-t_2|$ as required.

If we were in the last remaining case $2t_1 < s < 2t_2$, then simply by triangle inequality
\begin{align*}
|\beta_1(s) - \beta_2(s)| &\leq |\beta_1(s) - P_{t_1}| + |P_{t_1} - \beta_2(s)|
\\&= |\beta_1(s) - \beta_1(2t_1)| + |\varphi_0(2t_1) - \varphi_0(s)|
\\&\leq L|s-2t_1| + L|2t_1 - s|
\\&\leq 4L|t_1-t_2|.
\end{align*}
This shows that the boundary maps $\varphi_t$ satisfy the required Lipschitz estimates in $t$. Now we must still show that the same holds for the shortest curve extensions $H_{\varphi_t}$. The proof of this fact will still follow the same types of arguments as the proof of Lemma \ref{lem:shortestcurve}, but we must elaborate more as in this case we are dealing with two shortest curves within two different domains. However, we may again deduce the global Lipschitz-continuity from a local result and hence suppose that the considered time interval is small so that the geometry of the boundary is not too different between the two domains.

Let thus $z \in S$ and $t_1,t_2 \in (0,1/2)$. Let $l$ be the horizontal segment in $S$ which passes through $z$ and let $a$ and $b$ be its endpoints from left to right. Then $\varphi_{t_1}(a) = \varphi_{t_2}(a)$ as the mappings are the same on the left side. Let $\yy_{t_1}$ be the Jordan domain bounded by $\varphi_{t_1}(\partial S)$ and $L^{t_1}$ be the shortest curve within the closure of $\yy_{t_1}$ between $\varphi_{t_1}(a)$ and $\varphi_{t_1}(b)$. We define $\yy_{t_2}$ and $L^{t_2}$ analogously. We also let $p_{t_1} = \varphi_{t_1}(b)$ and $p_{t_2} = \varphi_{t_2}(b)$.


\begin{figure}
\phantom{a}
\vskip 150pt
{\begin{picture}(0.0,0.0) 
     \put(-220.2,0.2){\includegraphics[width=1.4\textwidth]{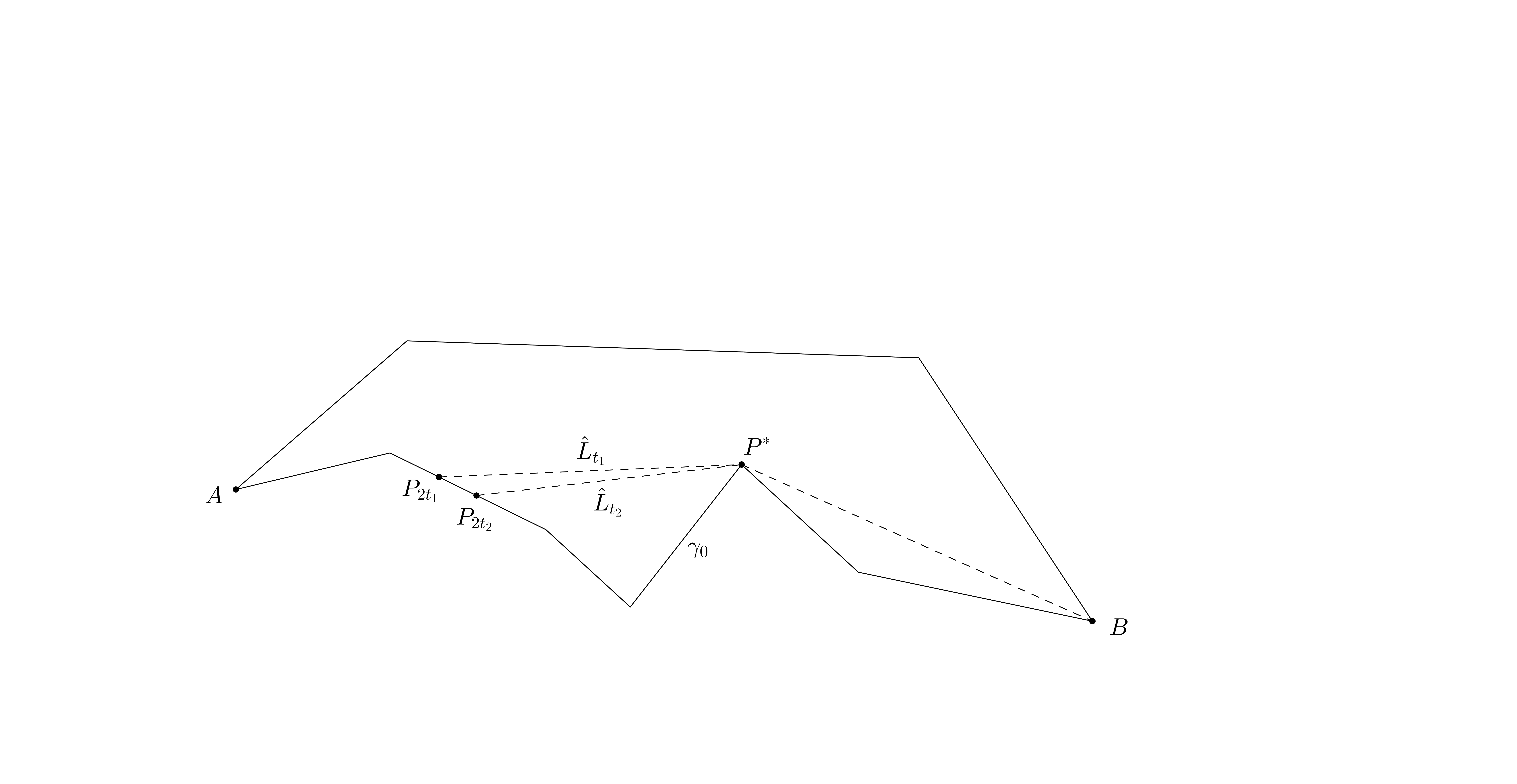}}
  \end{picture}
  }
\vskip -50pt	
\caption{For $t_1$ and $t_2$ close $P_{t_1}$ and $P_{t_2}$ are close and $\gamma_{t_1}$ and $\gamma_{t_2}$ are almost the same. 
The only difference is that segment $\hat{L}_{t_1}$ on $\gamma_{t_1}$ is replaced by two segments $P_{t_1}P_{t_2}$ and $\hat{L}_{t_2}$.}\label{fig:navic}
\end{figure}

We now use locality and consider $t_1$ fixed while choosing $t_2$ close enough to $t_1$ so that $P_{t_1}$ and $P_{t_2}$ lie on the same segment of the piecewise linear curve $\gamma_0$ (see Fig. \ref{fig:navic}). Note that the curve $\gamma_{t_1} = \varphi_{t_1}(s_+)$ is also piecewise linear, it consists of a part of $\gamma_0$ from $A$ to $P_{t_1}$ and a shortest curve from $P_{t_1}$ to $B$ which we shall call $\beta_{t_1}$. Let $\hat{L}_{t_1}$ be the line segment of $\beta_{t_1}$ which starts from $P_{t_1}$. We define $\beta_{t_2}$ and $\hat{L}_{t_2}$ similarly. Then if $t_2$ is chosen sufficiently close to $t_1$, the endpoint of $\hat{L}_{t_2}$ must lie on $\hat{L}_{t_1}$, let's call this endpoint $P^*$. This means that the only difference between the curves $\gamma_{t_1}$ and $\gamma_{t_2}$ is the following. If $t_2 > t_1$, the curve $\gamma_{t_1}$ travels in a single line segment from $P_{t_1}$ to $P^*$ while the curve $\gamma_{t_2}$ travels between the same points in two line segments $P_{t_1}P_{t_2}$ and $P_{t_2}P^*$. From $A$ to $P_{t_1}$ and $P^*$ to $B$ the curves are the same. If $t_2 < t_1$ the same happens but with $t_1$ and $t_2$ interchanged.

This simplification of the difference in geometry between $\gamma_{t_1}$ and $\gamma_{t_2}$ helps us with the next part, which is to consider the relation between the curves $L^{t_1}$ and $L^{t_2}$. We split the argument into a few cases.
\\\\
\textbf{Case 1.} If $p_{t_1}=\varphi_{t_1}(b)$ does not lie on the segment of $\gamma_{t_1}$ between $P_{t_1}$ and $P^*$.\\\\
In this case, $p_{t_1}$ lies on the common boundary of $\yy_{t_1}$ and $\yy_{t_2}$. We now define another map on the horizontal segment $l$ by considering the shortest curve from $\varphi_{t_1}(a)$ to $p_{t_1}$, but this time within the closure of $\yy_{t_2}$. Let this map be called $\Phi : l \to \bar{\yy_{t_2}}$ and parametrize it in constant speed also. Then the result of Lemma \ref{lem:shortestcurve} shows that $|H_{\varphi_{t_2}}(z) - \Phi(z)|$ may be estimated from above in terms of a constant times the length of the boundary of $\yy_{t_2}$ between $p_{t_1}$ and $p_{t_2}$. But the boundary estimates from before show that this length may be estimated from above by $CL|t_1-t_2|$. 

Hence due to the triangle inequality 
\[|H_{\varphi_{t_2}}(z) - H_{\varphi_{t_1}}(z)| \leq |H_{\varphi_{t_2}}(z) - \Phi(z)|+ |\Phi(z) - H_{\varphi_{t_1}}(z)|\]
it remains to consider the quantity $|\Phi(z) - H_{\varphi_{t_1}}(z)|$. This quantity depends on the curves $L^{t_1}$ and $\Phi(l)$. These curves are both shortest curves from $\varphi_{t_1}(a)$ to $p_{t_1}$. However, one is within the domain $\yy_{t_1}$ and the other is within the domain $\yy_{t_2}$. Thus we are to investigate how this change of domain affects the behaviour of the shortest curve.
\begin{figure}[H]\label{fig:short1}
\includegraphics[scale=0.4]{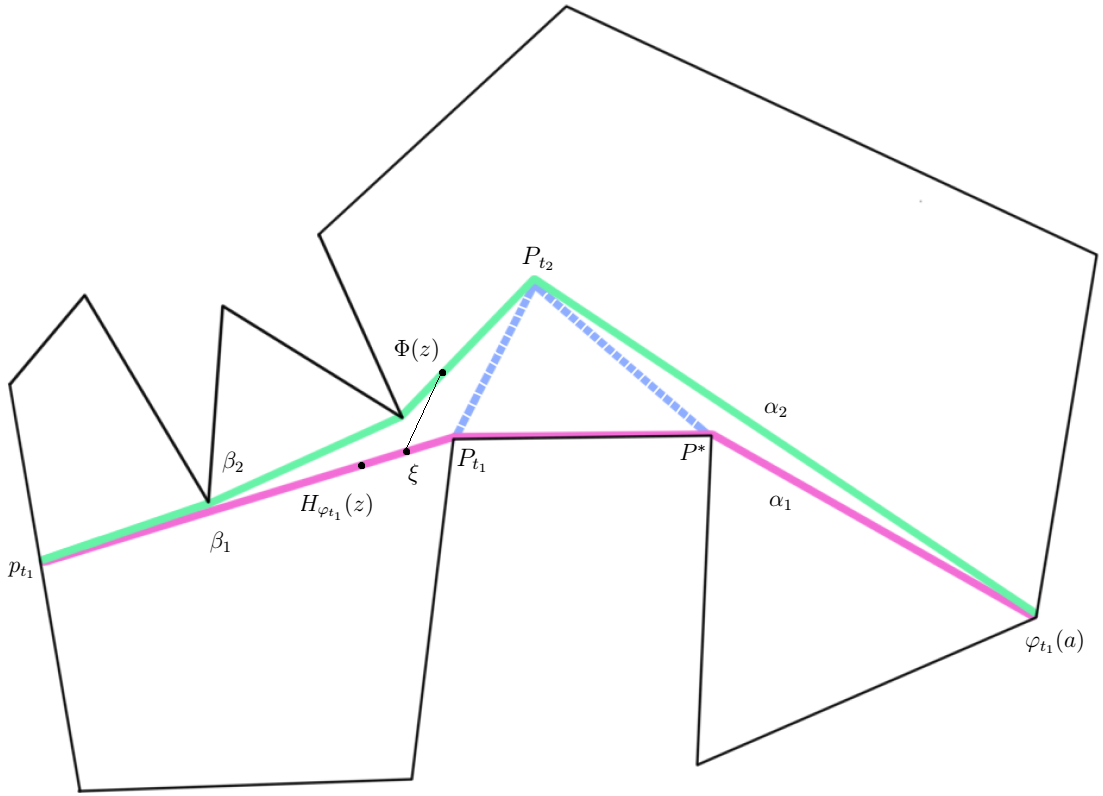}
\caption{Case 1: Two shortest curves between $\varphi_{t_1}(a)$ to $p_{t_1}$ in different domains. The boundary of $\mathbb{Y}_{t_1}$ is denoted by the black piecewise linear curve. The domain $\mathbb{Y}_{t_2}$ is created from $\mathbb{Y}_{t_1}$ by adding a triangle $\Delta P_{t_1}P_{t_2}P^*$. }\label{sestka}
\end{figure}
\noindent \emph{Case 1a.} Suppose that the curve $L^{t_1}$ does not touch the segment $P_{t_1}P^*$.

Since $L^{t_1}$ is the shortest curve between $\varphi_{t_1}(a)$ and $\varphi_{t_1}(b)$ in $\yy_{t_1}$, if $\yy_{t_2} \subset \yy_{t_1}$ then $\Phi(l)$ (the shortest curve between the same points in $\yy_{t_2}$) must be at least as long as $L^{t_1}$. But since $L^{t_1}$ does not intersect $P_{t_1}P^*$ we must have $L^{t_1} \subset \yy_{t_2}$ and thus $L^{t_1} = \Phi(l)$. If $\yy_{t_2}$ is not contained in $\yy_{t_1}$, which is when $P_{t_2}$ lies outside of $\yy_{t_1}$, then it still must hold that $L^{t_1} = \Phi(l)$ because the shortest curve $\Phi(l)$ cannot pass through the interior the triangle $\Delta P^* P_{t_1} P_{t_2}$ as it can only enter and exit through the segment $P^* P_{t_1}$. Thus there is nothing to prove in this case.\\\\
\noindent\emph{Case 1b.} Suppose that $P_{t_1} \in L^{t_1}$ and $P_{t_2} \in \Phi(l)$.

Let the part of $L^{t_1}$ between $\varphi_{t_1}(a)$ and $P_{t_1}$ be called $\alpha_1$ and the part from $P_{t_1}$ to $p_{t_1}$ be called $\beta_1$. Similarly, the part of $\Phi(l)$ from $\varphi_{t_1}(a)$ to $P_{t_2}$ is $\alpha_2$ and from $P_{t_2}$ to $p_{t_1}$ is $\beta_2$. Let $|P_{t_1} - P_{t_2}| = \delta$.

Let us say that a curve in $\bar{\yy_1}$ does not cross the segment $P_{t_1}P_{t_2}$ if that curve is a uniform limit of curves within $\yy_1 \setminus P_{t_1}P_{t_2}$. Note that none of the curves $\alpha_1, \alpha_2, \beta_1$ and $\beta_3$ pass through the interior of the triangle $\Delta P_{t_1}P_{t_2}P^*$ and also do not cross the segment $P_{t_1}P_{t_2}$. Hence within the class of curves in $\bar{\yy_1}$ which do not cross the segment $P_{t_1}P_{t_2}$, these curves are also the shortest curves between their respective endpoints.

We suppose that $\Phi(z)$ is on $\beta_2$. The case where it is on $\alpha_2$ is proven similarly. We define a point $\xi \in \beta_1$ as the intersection point of $\beta_1$ with the line passing through $\Phi(z)$ and parallel to $P_{t_1} P_{t_2}$ (see Fig \ref{sestka}). Due to the fact that $\beta_1$ and $\beta_2$ are shortest curves in $\bar{\yy_1}$ which do not cross the segment $P_{t_1}P_{t_2}$, the segment from $\Phi(z)$ to $\xi$ lies entirely between these two curves and has length smaller than $\delta$ - this can be argued similarly as the convexity part in Case 2 of Lemma \ref{lem:shortestcurve}. Let $\beta_2^*$ be the part of $\beta_2$ from $p_{t_1}$ to $\Phi(z)$ and $\beta_1^*$ be the part of $\beta_1$ from $p_{t_1}$ to $\xi$. Then a simple shortest curve estimate shows that
\begin{equation}\label{eq:betastar}||\beta_2^*| - |\beta_1^*|| \leq |\Phi(z)-\xi| \leq \delta.
\end{equation}
Similarly we may find that
\begin{equation}
\begin{split}
||\beta_2|-|\beta_1|| &\leq \delta\\
||\alpha_1|- |\alpha_2|| &\leq \delta.
\end{split}\label{eq:beta_alpha}
\end{equation}

Now consider the length of the part of $H_{\varphi_1}(l)$ between $p_{t_1}$ and $H_{\varphi_{t_1}}(z)$, call this length $\tau$. Due to constant speed parametrization, if the distance from $a$ to $z$ is $x$, we find that $\tau = (|\alpha_1| + |\beta_1|)x/|l|$. But since $x \leq |l|$ and the estimates \eqref{eq:beta_alpha}, we find that
\[||\tau - |\beta_2^*|| = \left||\tau - \frac{(|\alpha_2| + |\beta_2|)x}{|l|}\right| \leq 2\delta.\]
However, \eqref{eq:betastar} then implies that $|\tau - |\beta_1^*|| \leq 3\delta$. This further gives that $|\xi-H_{\varphi_{t_1}}(z)| \leq 3\delta$ and finally $|\Phi(z) - H_{\varphi_{t_1}}(z)| \leq 4\delta$. Since $\delta \leq CL|t_1-t_2|$ this is enough.\\\\
\noindent \emph{Case 1c.} Suppose $P_{t_2} \in \Phi(l)$, $P_{t_1} \notin L^{t_1}$ but either $L^{t_1}$ passes through $P_{t_1}P_{t_2}$ or through $P^* P_{t_2}$.

If $L^{t_1}$ passes through $P_{t_1}P_{t_2}$, let the intersection point of $P_{t_1}P_{t_2}$ and $L^{t_1}$ be $Q$. This case can be handled the same way as Case 1b, with $Q$ taking the role of $P_{t_1}$. The case where $L^{t_1}$ passes through $P^* P_{t_2}$ can be handled symmetrically.
\\\\
\noindent \emph{Case 1d.} Suppose that $P_{t_2} \notin \Phi(l)$.

This case appears either when the point $P_{t_2}$ is outside the domain $\yy_{t_1}$ or when $L^{t_1}$ only passes through the triangle $\Delta P_{t_1} P_{t_2} P^*$ at one of the vertices $P_{t_1}$ or $P^*$ (See Figure \ref{fig:short5}). In all of these cases the curves $L^{t_1}$ and $\Phi(l)$ are the same, and there is nothing to prove. This handles all the possible options and finishes the proof of Case 1.
\begin{figure}[H]\label{fig:short5}
\includegraphics[scale=0.4]{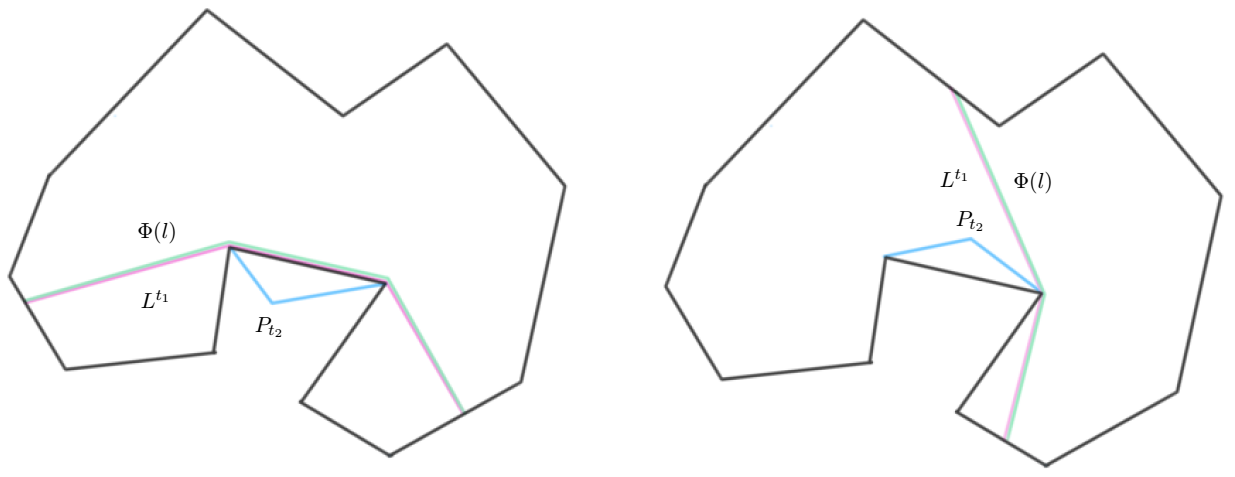}
\caption{Case 1d: Reduces to pictured possibilities in which the curves $L^{t_1}$ and $\Phi(l)$ are the same.}
\end{figure}
\textbf{Case 2.} If $p_{t_2}$ does not lie on the segment of $\gamma_{t_2}$ between $P_{t_2}$ and $P^*$. This case may be treated with the same arguments as Case 1, with $t_1$ and $t_2$ interchanged.\\\\
\textbf{Case 3.} We suppose that $p_{t_1}$ lies on the segment $P_{t_1}P^*$ and $p_{t_2}$ on the segment $P_{t_2}P^*$.\\\\
By symmetry, suppose that $t_1 < t_2$. We now consider the triangle $T = \Delta P_{t_1}P_{t_2}P^*$, but must split into cases depending on if this triangle is inside or outside of $\yy_{t_1}$.

\emph{Case 3a.} If $T$ is inside of $\yy_{t_1}$. The shortest curve $L^{t_1}$ must pass through $T$ before it reaches its endpoint at $p_{t_1}$. Moreover, the part of $L^{t_1}$ inside the closure of $T$ must be a single segment since $T$ is convex. Now, the point $p_{t_2}$ splits the union of the segments $P_{t_1}P_{t_2}$ and $P_{t_2} P^*$ into two parts. Let $\hat{\gamma}$ be the part which does not intersect $L^{t_1}$.
\begin{figure}[H]\label{fig:short3}
\includegraphics[scale=0.35]{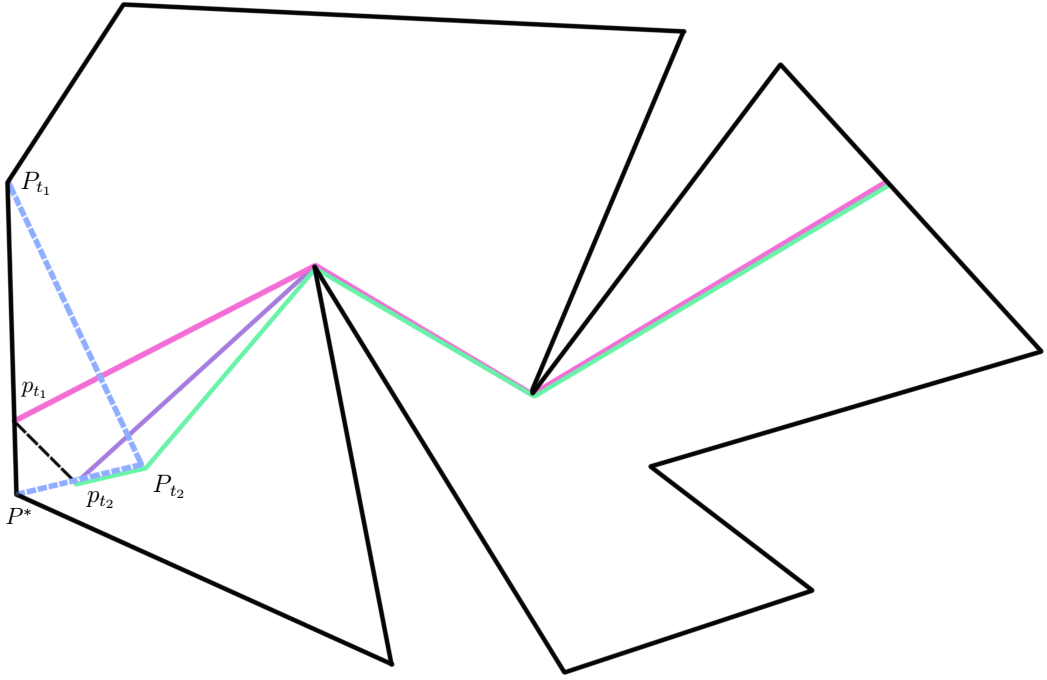}
\caption{Case 3a: Shortest curves to $p_{t_1}$ and $p_{t_2}$ when $T$ is inside of $\yy_{t_1}$. In this case, $\hat{\yy}$ is obtained by taking $\partial \yy_1$ and replacing $p_{t_1}P^*$ with $p_{t_1}p_{t_2}$ and $p_{t_2}P^*$. Again $\mathbb{Y}_{t_2}$ is created from $\mathbb{Y}_{t_1}$ by adding a triangle $\Delta P_{t_1}P_{t_2}P^*$.}
\end{figure}

The idea now is to create a new domain $\hat{\yy}$. We take the Jordan curve $\partial \yy_{t_1}$, add the union of $p_{t_1}p_{t_2}$ and $\hat{\gamma}$ to it, and remove the segment of $\partial \yy_{t_1}$ which has the same endpoints as this union does (either we remove $p_{t_1}P^*$ or $p_{t_1}P_{t_1}$). This Jordan curve now defines $\hat{\yy}$. An equivalent definition is to cut off from $\yy_{t_1}$ a region bounded by $p_{t_1}p_{t_2}$ and $\hat{\gamma}$. The key point is that by this construction the curve $L^{t_1}$ still lies in the closure of $\hat{\yy}$. In fact, the curve $L^{t_1}$ is still the shortest curve from $\varphi_{t_1}(a)$ to $p_{t_1}$ within the new domain $\hat{\yy}$. This is due to the fact that the shortest curve from $\varphi_{t_1}(a)$ to $p_{t_1}$ does not change if we remove a region of the domain which does not intersect this shortest curve to begin with.

Let now $\Phi : l \to \bar{\hat{\yy}}$ denote the shortest curve from $\varphi_{t_1}(a)$ to $p_{t_2}$ in the closure of $\hat{\yy}$, parametrized with constant speed. Now we split our estimates via the triangle inequality
\[|H_{\varphi_{t_2}}(z) - H_{\varphi_{t_1}}(z)| \leq |H_{\varphi_{t_2}}(z) - \Phi(z)|+ |\Phi(z) - H_{\varphi_{t_1}}(z)|.\]
The quantity $|\Phi(z) - H_{\varphi_{t_1}}(z)|$ may now be estimated via the arguments of Lemma \ref{lem:shortestcurve}, since both $\Phi$ and $H_{\varphi_{t_1}}$ map the horizontal segment $l$ to a shortest curve within $\hat{\yy}$, and the distance between their endpoints $p_{t_2}$ and $p_{t_1}$ is estimated from above by $CL|t_1-t_2|$.

The quantity $|H_{\varphi_{t_2}}(z) - \Phi(z)|$ is dealt with the same arguments as Case 1, since $\Phi$ and $H_{\varphi_{t_2}}$ map the horizontal segment $l$ to shortest curves from $\varphi_{t_1}(a)$ to $p_{t_2}$, however in different domains $\hat{\yy}$ and $\yy_{t_2}$. The difference between these domains is again small.
\\\\
\emph{Case 3b.} If $T$ is outside of $\yy_{t_1}$. This case is handled much the same as the previous one, only now we create $\hat{\yy}$ from $\yy_{t_2}$ by adding $p_{t_1}p_{t_2}$ and the part of $P_{t_1}P^*$ which does not intersect $L^{t_2}$. We also remove either $p_{t_2}P^*$ or the two segments of $\partial \yy_2$ which join $p_{t_2}$ with $P_{t_1}$ to create the Jordan curve that bounds $\hat{\yy}$. Now the situation is dealt with the same arguments as the previous case.
\\\\
\emph{Finishing the proof.} It remains to make a slight modification to the curves $\gamma_t$ to make them mutually nonintersecting (here we exclude intersection at the endpoints $A$ and $B$) and to make sure that this does not interfere with the claimed estimates. Note that if two of these curves do intersect, they must do so at a vertex $P$ of $\partial \hat{\yy}$, which was the Jordan domain bounded by $\gamma_0$ and $\gamma_1$. At any such vertex $P$ we attach to it a small segment $PV_P$ facing the interior of $\hat{\yy}$ and bisecting the angle of $\partial \hat{\yy}$ at $P$.

Now for each such segment we consider all the curves $\gamma_t$ which pass through $PV_P$ and let the intersection point of $\gamma_t$ with this segment be $P_t$. Thus for those parameters $t$ the map $t \to P_t$ defines either an increasing or decreasing parametrization of $P V_P$, which is not strictly monotone as some interval of parameters is sent to the point $P$. However, we may make an arbitrarily small modification to this parametrization to make it strictly monotone, replacing each point $P_t$ with another point $P_t^*$ on $PV_P$.

This gives us a way to modify each of the piecewise linear curves $\gamma_t$ by another curve $\gamma_t^*$ which, for each segment $PV_P$ that intersects $\gamma_t$, passes through the point $P_t^*$ instead of $P_t$. As this modification may be done in an arbitrarily small way we may assume that the Lipschitz estimates we obtained before for $\varphi_t$ and for $H_{\varphi_t}$ also hold after the modification up to a multiplicative constant arbitrarily close to $1$. It is also not difficult to see that the curves $\gamma_t^*$ are now mutually noinintersecting, for further details see Section \ref{sec:injective} where a similar construction is explained in more depth.
\end{proof}

\section{The 3D extension}\label{sec:3d}

\begin{proof}[Proof of Theorem \ref{thm:mainsum}]
We now describe the process of extending a given homeomorphic boundary map $\varphi : \rr^2 \to \rr^2$ locally as a homeomorphism of the upper half space to itself. Recalling that $S_0 = [0,1]^2$ is the unit square in the plane, our aim is to define a continuous injective extension $h : [0,1]^3 \to \rr^3_+$ which agrees with $\varphi$ on $[0,1]^2 \times \{0\}$ (this is identified with $S_0$). However, first we will define a monotone extension $h : [0,1]^3 \to \rr^3_+$ using the two-dimensional shortest curve extensions defined before, and in the last section we will explain how this extension can be modified to be homeomorphic.

The idea is to decompose the domain space $[0,1]^3$ dyadically into cubes $U_{k,j}$. Recall the original standard dyadic decomposition of $S_0$ into dyadic squares $\tilde{Q}_{k,j}$. We define $U_{k,j} = \tilde{Q}_{k,j} \times [2^{-k}, 2^{-(k-1)}]$. Thus $U_{k,j}$ is a cube of side length $2^{-k}$ and the union of all such cubes decompose the domain space $[0,1]^3$. The idea is to map each cube to a 'cylindrical' region $V_{k,j}$.

Recall that the curve $\Gamma_{k,j}$, as defined in Lemma \ref{lem:piecewiselinear}, denotes a piecewise linear replacement of the image curve $\varphi(\partial\tilde{Q}_{k,j})$. We define the top face of $V_{k,j}$ as the horizontal region bounded by the curve $\Gamma_{k,j} \times \{2^{-(k-1)}\}$. On the next dyadic level, let $\hat{\Gamma}_{k,j}^{(m)}$ for $m = 1,2,3,4$ denote the four piecewise linear curves of the form $\Gamma_{k',j'}$ for some $k',j'$ which are obtained from the images of the four dyadic children of $\tilde{Q}_{k,j}$. Moreover, let $\hat{\Gamma}_{k,j}$ denote the piecewise linear Jordan curve which corresponds to the outer boundary of the union of all four $\hat{\Gamma}_{k,j}^{(m)}$. Then the bottom face of $V_{k,j}$ will be defined as the horizontal region bounded by the curve $\hat{\Gamma}_{k,j} \times \{2^{-k}\}$. See Figure \ref{fig:grid1}.

We aim to define the extension $h$ so that it keeps horizontal planes fixed, meaning that $[0,1]^2 \times \{t\}$ is mapped to $\rr^2 \times \{t\}$ for each $t > 0$. In terms of the sets $U_{k,j}$ and $V_{k,j}$, the map $h$ will map each horizontal section of $U_{k,j}$ to the horizontal section of $V_{k,j}$ of the same height. The horizontal sections of $V_{k,j}$ will still need to be defined, however, and to do this we will need to construct an appropriate homotopy between the curves $\Gamma_{k,j}$ and $\hat{\Gamma}_{k,j}$. Before we begin the construction, we state our main goal in terms of estimates as the following.
\\\\
\textbf{Goal:} The map $h : U_{k,j} \to V_{k,j}$ will be a Lipschitz mapping. The Lipschitz constant of the map should be estimated from above by a uniform constant times the quantity $(|\Gamma_{k,j}| + \sum_{m=1}^4|\hat{\Gamma}_{k,j}^{(m)}|)2^k$, or possibly this quantity added together with the same quantity over all of the neighbours of $\Gamma_{k,j}$.\\\\

\begin{figure}[H]
\includegraphics[scale=0.4]{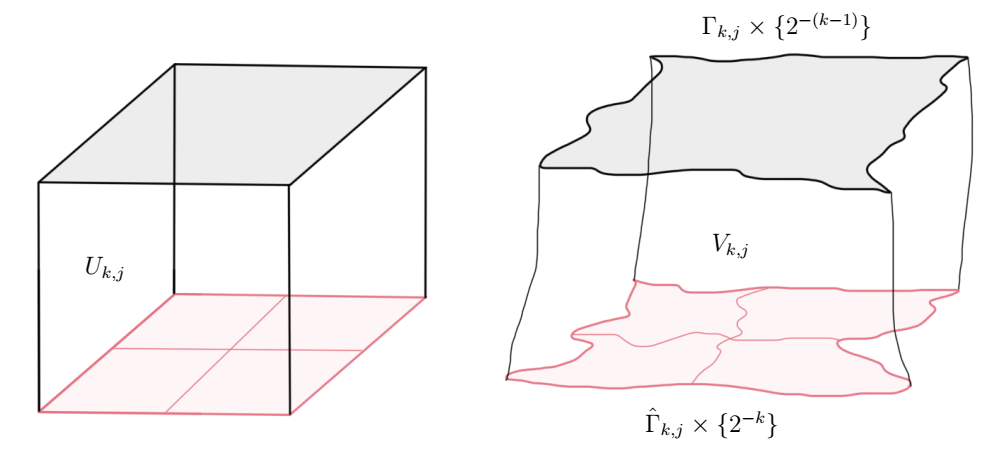}
\caption{The cube $U_{k,j}$ and its image set $V_{k,j}$ defined as a region spanned by the curve $\Gamma_{k,j} \times \{2^{-(k-1)}\}$ and its corresponding curve $\hat{\Gamma}_{k,j} \times \{2^{-k}\}$ on the next level.}\label{fig:grid1}
\end{figure}

The process of deducing the statement of Theorem 1.2 from these results is as follows. After this section we have defined the monotone extension $h$ on each dyadic cube $U_{k,j}$ so that the goal estimate above holds, and this extension is further modified into an injective extension $h$ in Section \ref{sec:injective} with the same estimates. In Theorem 1.2, the boundary map $\varphi$ is defined not on the plane but on the unit sphere, but this topological difference poses no additional difficulty to applying the same extension technique. Indeed, we may decompose the unit sphere dyadically and then apply the same extension process locally with the up direction replaced with the inward normal direction. The goal estimate above can then be used to estimate the Sobolev norm of the extension $h$ inside each dyadic region $U_{k,j}$ by estimating the differential $|Dh|$ above by the Lipschitz constant. Combined with the goal estimate this gives
\[\int_{U_{k,j}} |Dh(z)|^q \, dz \, \leq 2^{k(q-3)} (|\Gamma_{k,j}| + \sum_{m=1}^4|\hat{\Gamma}_{k,j}^{(m)}|)^q.\]
Summing up over all $U_{k,j}$ and recalling that the curves $\Gamma_{k,j}$ were defined as piecewise approximations of $\varphi(\partial \tilde{Q}_{k,j})$ with comparable length gives the statement of Theorem \ref{thm:mainsum}. \\\\

We now proceed to the construction of $h$.

\emph{Step 1.} We define $h$ on the sides of the top and bottom faces of $U_{k,j}$. We wish to map the top sides $\partial \tilde{Q}_{k,j} \times \{2^{-(k-1)}\}$ to the Jordan curve $\Gamma_{j,k}$ and the bottom sides $\partial \tilde{Q}_{k,j} \times \{2^{-k}\}$ to $\hat{\Gamma}_{j,k}$. Note that here and what follows we abuse $\partial$ to mean the 1D boundary of these sets rather than taking the topological boundary of the sets in 3D space.

\emph{Step 2.} We define $h$ on the top and bottom faces of $U_{k,j}$. To simplify notation, we set $\mathcal{U}_t = \tilde{Q}_{k,j} \times \{t\}$. Furthermore, let $top := 2^{-(j-1)}$ and $bot := 2^{-k}$ so that $\mathcal{U}_{top}$ is the top face and $\mathcal{U}_{bot}$ is the bottom one. Similarly we set $\varphi_t = h\vert_{\partial \mathcal{U}_{t}}$ and $h_{t} = h\vert_{\mathcal{U}_t}$, although only $\varphi_{top}$ and $\varphi_{bot}$ have been defined so far. On $\mathcal{U}_{top}$, we simply define $h_{top}$ as the shortest curve extension of $\varphi_{top}$. Note that this choice also forces us to define $h_{bot}$ on $\mathcal{U}_{bot}$ in a specific way to avoid discontinuity. Indeed, the bottom side $\mathcal{U}_{bot}$ is in fact the union of four top sides of dyadic cubes of the form $U_{k+1,j'}$ on the next level. Thus on $\mathcal{U}_{bot}$ the map $h_{bot}$ is defined separately in each of the four squares as the shortest curve extension of the corresponding boundary values.

\emph{Step 3.} Let $mid := 2^{-k} + 2^{-k-1}$ be the middle point of $[2^{-k},2^{-(k-1)}]$ so that $\mathcal{U}_{mid}$ is the middle level of the cube $U_{k,j}$. On the sides of $\mathcal{U}_{mid}$ and for every parameter $t \in [mid,bot]$, we define $\varphi_t$ equal to $\varphi_{bot}$. On $\mathcal{U}_{mid}$ we define $h_{mid}$ as the shortest curve extension of $\varphi_{mid}$. Hence for $t \in [mid,bot]$, the mapping $h_t$ has the same boundary values on each level $\mathcal{U}_t$ but is a different map on the faces $\mathcal{U}_{mid}$ and $\mathcal{U}_{bot}$. We return to this part in a later step and describe how to define $h_t$ for $t \in (mid,bot)$ to give the correct isotopy between the maps $h_{mid}$ and $h_{bot}$.

\emph{Step 4.} For $t \in [top,mid]$, we will define $h_t$ as the shortest curve extension of $\varphi_t$. However, we have not yet defined $\varphi_t$ for these parameters. Note that the image of $\varphi_{top}$ is $\Gamma_{k,j}$ and the image of $\varphi_{mid}$ is $\hat{\Gamma}_{j,k}$. Thus we must define a homotopy $\varphi_t$ between these two curves which is what we will do now.

\begin{figure}[H]
\includegraphics[scale=0.4]{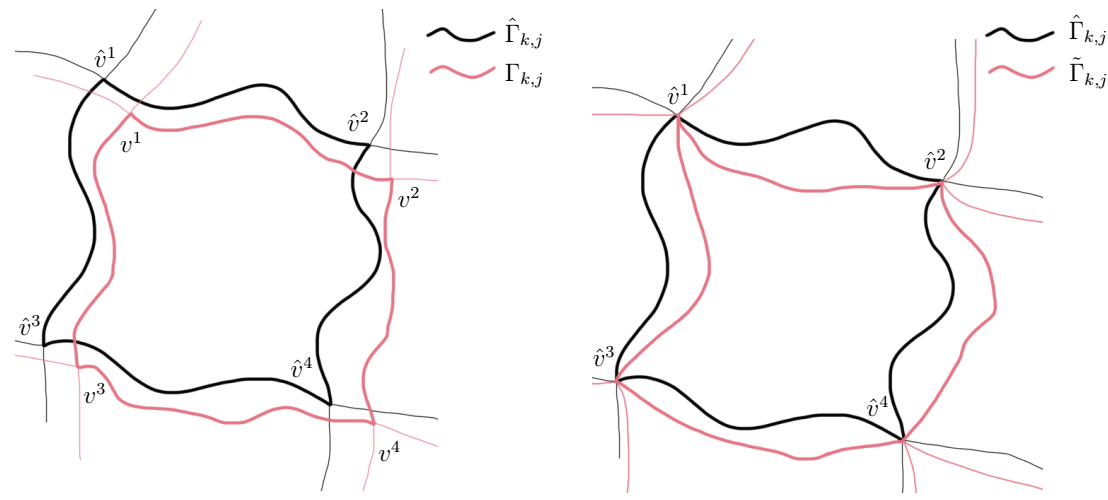}
\caption{On the left, the curve $\Gamma_{k,j}$ and its corresponding curve $\hat{\Gamma}_{k,j}$ on the next level. On the right, $\Gamma_{k,j}$ has been modified to $\tilde{\Gamma}_{k,j}$.}\label{fig:grid2}
\end{figure}

The left part of Figure \ref{fig:grid2} depicts the curves $\Gamma_{k,j}$ and $\hat{\Gamma}_{k,j}$. As in the figure, let us label the vertices of these curves by $v^j$ and $\hat{v}^j$, $j = 1,2,3,4$ in corresponding order. We pick one pair of such vertices, say $v_1$ and $\hat{v}_1$. The vertex $v_1$ is the intersection point of two sides of the curve $\Gamma_{k,j}$ as well as two other sides of curves in the same grid, for a total of four. We let the midpoint of these sides be $m_j$, $j = 1,2,3,4$, see Figure \ref{fig:grid4}. We similarly define four points $\hat{m}_j$ as the midpoints of the segments of the grid of curves $\tilde{\Gamma}_{k,j}$ which meet at $\hat{v}_1$, numbered correspondingly to the points $m_j$. We now connect each of the points $m_j$ with $\hat{m}_j$ through a piecewise linear curve $\alpha_j$ which does not intersect either of the grids and has length comparable to the infimal length of such curves. Travelling along the curves $\alpha_j$ and the two grids, we let $K_1$ denote the area bounded by the points $\hat{m}_1, m_1, v_1, m_2, \hat{m}_2$ and $\hat{v_2}$. Similarly we define $K_2$ as the area bounded by $m_4,\hat{m}_4,\hat{v}_1,\hat{m}_3,m_3$ and $v_1$.

\begin{figure}[H]
\includegraphics[scale=0.5]{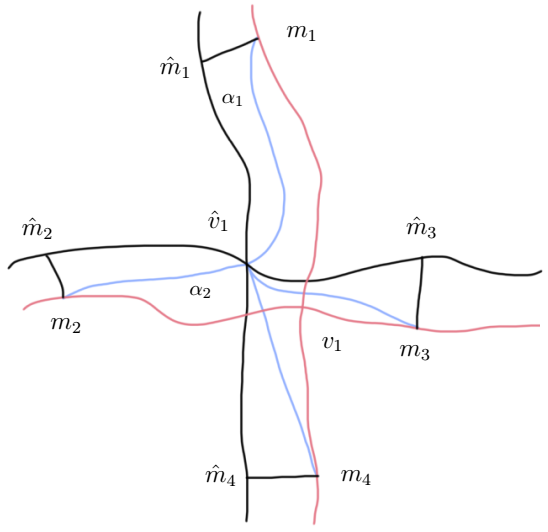}
\caption{The plus-shaped region which is the union of the sets $K_1$ and $K_2$.}\label{fig:grid4}
\end{figure}

Our aim now is to deform the cross formed by the curves with endpoints at $m_1,\ldots,m_4$ and intersecting at $v_1$, to a cross with the same endpoints but middle point at $\hat{v}_1$ instead. This deformation should be done as a homotopy in $t$ with controlled Lipschitz estimates just as we have done before. Moreover, we wish to introduce no new intersection points during this homotopy and keep the deformation entirely within $K_1$ and $K_2$.

We first connect the points $m_1$ and $\hat{v}_1$ with a piecewise linear Jordan curve $\alpha_1$ which does not intersect any of the other considered curves and has distance comparable to the sum of the length of the curves from $m_1$ to $\hat{m}_1$ and $\hat{m}_1$ to $\hat{v}_1$. This can be done for example by choosing a curve sufficiently close to those two curves but not intersecting them or itself. Similarly we define a curve $\alpha_2$ from $m_2$ to $\hat{v}_1$, see again Figure \ref{fig:grid4}.

Let $\psi_0$ be the union of the curves from $m_1$ to $v_1$ and from $v_1$ to $m_2$, parametrized on $[0,1]$. Similarly let $\psi_1$ be the union of $\alpha_1$ and $\alpha_2$. Now by Lemma \ref{lem:changeboundary} there must be a homotopy $\psi_t$ between $\psi_0$ and $\psi_1$ with the correct Lipschitz estimates in $t$ and such that $\psi_t$ lies between $\psi_0$ and $\psi_1$ (hence within $K_1$).

Moreover, by reparametrizing such a homotopy at the endpoints if needed (Lemma \ref{lem:sameboundary}) we may suppose that $\psi_0(1/2) = v_1$ and $\psi_1(1/2) = \hat{v}_1$. We then let $\Psi$ denote the curve given by $t \mapsto \psi_t(1/2)$. Due to the fact that the curves $\psi_t$ coming from Lemma \ref{lem:changeboundary} do not intersect we see that the curve $\Psi$ is piecewise linear and nonintersecting. It also connects $v_1$ to $\hat{v}_1$ within $K_1$.

The deformation from $\psi_0$ to $\psi_1$ gives one part of the sought homotopy between the two crosses. Let $\beta_1$ denote the curve from $m_3$ to $v_1$ and $\beta_2$ the curve from $m_4$ to $v_1$. We denote by $\psi_0^*$ the union of $\beta_1$ and $\beta_2$, parametrized again on $[0,1]$. It remains to show that we can construct a homotopy $\psi_t^*$ so that $\psi_t^*(0) = m_3$, $\psi_t^*(1/2) = \psi_t(1/2)$, $\psi_t^*(1) = m_4$, and so that the curve $\psi_t^*$ has no additional intersection points with $\psi_t$ nor $\partial K_2$.

We first construct a homotopy from $\psi_0^*$ to a curve $\tilde{\psi}_1$ which is obtained by traveling the curve $\beta_1$, then the curve $\Psi$, then back along the curve $\Psi$ in reverse, and finally along $\beta_2$. Supposing that $\beta_1, \beta_2$ and $\Psi$ are initially parametrized on $[0,1]$, we let

\begin{equation*}
\tilde{\psi}_t(s) = \begin{cases}
  \beta_1((1/2-t/4)^{-1}s) & 0 \leq s \leq 1/2-t/4\\
  \Psi(4s-2+t) & 1/2 - t/4 < s \leq 1/2\\
  \Psi(t-4s+2) & 1/2 < s \leq 1/2 + t/4 \\
  \beta_2((1/2-t/4)^{-1}(s-1/2-t/4)) & 1/2 + t/4 < s \leq 1
  \end{cases} \\
\end{equation*}
This curve travels first along $\beta_1$, then $\Psi([0,t])$, then $\Psi([0,t])$ backwards and then along $\beta_2$. Thus it gives the desired homotopy and it is easy to verify that such a homotopy satisfies the required Lipschitz estimates in $s$ and $t$.

This homotopy otherwise would suit our purposes but it obviously intersects itself so we cannot use it as the definition for $\psi_t^*$. Instead, we will define $\psi_t^*$ as the following modification of $\tilde{\psi}_t$. Essentially as we are travelling the curve $\Psi([0,t])$ twice in $\tilde{\psi}_t$, we wish to instead first travel a curve very close to $\Psi([0,t])$ and with the same endpoint $\Psi(t)$, and then travel back along another curve very close to $\Psi([0,t])$ but on the other side of the first curve so that we are never intersecting ourselves. Thus we are, in a sense, opening up the curve $\Psi$ into two curves. See Figure \ref{fig:open1}.

To define such a process, for each of the vertices $P \neq \hat{v}_1$ on the piecewise linear curve $\Psi$, we associate a very small segment $S_P$ with center at $P$ so that locally $P$ divides $S_P$ into two segments which lie on separate sides of $\Psi$. The curve $\Psi$ divides each segment $S_P$ into two segments $S_P^+$ and $S_P^-$. We make the choice of labelling in such a way that the segments $S_P^+$ all lie on the same side of $\Psi$. For each $t$ we then define a curve $\Psi_t^+$ as follows.

For each vertex $P$ of $\Psi$ let $t_P$ be such that $\Psi(t_P) = P$. Then we count all such vertices $P$ for which $t_P \leq t$. For any such vertex, we pick a point $Q_{t,P}$ on $S_P^+$ so that $|Q_{t,P} - P|/|S_P^+| = (t-t_P)/(1-t_P)$. This means that at $t = t_P$, the point $Q_{t,P}$ is exactly $P$ and at $t = 1$ we arrive at the other endpoint of $S_P^+$. We then connect the points $Q_{t,P}$ for all $P$ such that $t_P \leq t$, together with the point $\Psi(t)$, to form a piecewise linear Jordan curve $\Psi_t^+$ which lies locally on one side of $\Psi$. This curve is then parametrized on $[0,t]$. A similar process on the other side defines a curve $\Psi_t^-$ so that the curves $\Psi_t^+$ and $\Psi_t^-$ do not intersect. Moreover, one may verify that both of the maps $s \to \Psi_t^\pm(s)$ satisfy the correct Lipschitz estimates in $(s,t)$ for sufficiently small choices of segments $S_P$ (which may be chosen arbitrarily small).

We then modify the curve $\tilde{\psi}_t$ as follows. Instead of travelling the whole curve $\beta_1$ first, we travel along $\beta_1$ until we get to its last segment which ends on $v_1$, and instead of going along the segment to $v_1$ we go along a segment to $\Psi_t^+(0)$ instead. We then travel along $\Psi_t^+$ and travel backwards along $\Psi_t^-$ to $\Psi_t^-(0)$. We then travel back along $\beta_2$, but must first replace the first segment of $\beta_2$ which starts at $v_1$ by a segment which starts at $\Psi_t^-(0)$ instead. This defines the curve $\psi_t^*$, and the four parts which make up this curve are parametrized on the same four intervals in the definition of $\tilde{\psi}_t$ above.

\begin{figure}[H]
\includegraphics[scale=0.4]{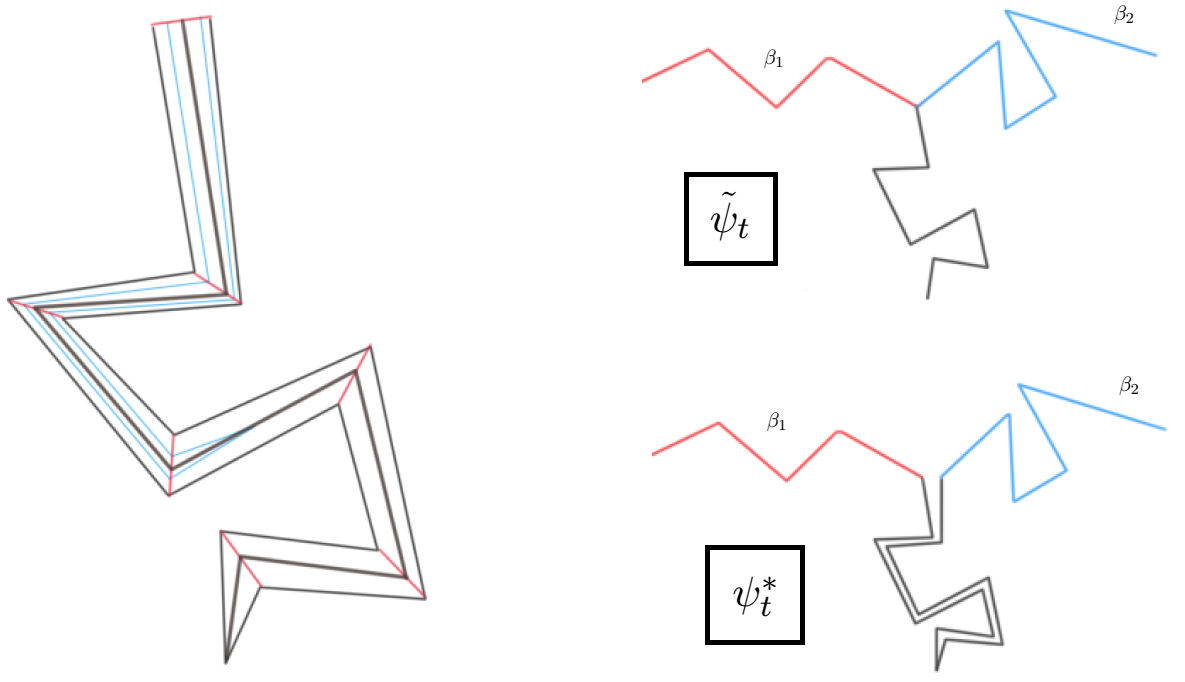}
\caption{Opening up the curve $\Psi$ to create a homotopy of Jordan curves.}\label{fig:open1}
\end{figure}

The curves $\psi_t^*$ are now non-intersecting and give the desired homotopy from $\psi_0$ to a curve $\psi_1^*$ which passes through $m_1,\hat{v}_1$ and $m_2$, lies entirely within $K_2$ and does not intersect itself or the other relevant curves. Thus in total we have defined a way to deform the cross with endpoints $m_1,\ldots,m_4$ and center $v_1$ to a cross with the same endpoints and center $\hat{v}_1$.

After doing this process for every vertex $v_j$ and every curve $\Gamma_{k,j}$ on level $k$, we have replaced the curve $\Gamma_{k,j}$ with another curve $\tilde{\Gamma}_{k,j}$ with the same vertices as $\hat{\Gamma}_{k,j}$ but not intersecting it, see Figure \ref{fig:grid2}. The homotopy between $\tilde{\Gamma}_{k,j}$ and $\hat{\Gamma}_{k,j}$ is now easy to construct. Between each pair of neighbouring vertices, say $\hat{v}_1$ and $\hat{v_2}$, we deform the part of $\tilde{\Gamma}_{k,j}$ into $\hat{\Gamma}_{k,j}$ via the method explained in Lemma \ref{lem:changeboundary}. After deforming each four parts in succession we have deformed $\tilde{\Gamma}_{k,j}$ into $\hat{\Gamma}_{k,j}$.

We provide a few more details regarding parametrization and estimates. We may divide the interval $[top,mid]$ into two equal parts, on one of which we deform $\Gamma_{k,j}$ into $\tilde{\Gamma}_{k,j}$ and on the other $\tilde{\Gamma}_{k,j}$ into $\hat{\Gamma}_{k,j}$. The first part may be further divided into four parts where we move each of the vertices $v^j$ to $\hat{v^k}$, and the second one depending on which part of $\tilde{\Gamma}_{k,j}$ we are deforming.

In the first part, the length of the relevant curves is always controlled from above by $|\Gamma_{k,j}| + |\hat{\Gamma}_{k,j}|$, plus the same quantity over the neighbours of $\Gamma_{k,j}$. As the initial curves are parametrized with constant speed we know by Lemma \ref{lem:changeboundary} that the Lipschitz-constant of the shortest curve extension $h$ in the $(z,t)$-variables is thus controlled by $2^k(|\Gamma_{k,j}| + |\hat{\Gamma}_{k,j}|)$ added with this quantity over the neighbours. In the second part, we are again using Lemma \ref{lem:changeboundary} and therefore the Lipschitz-constant is estimated from above by $2^k(|\Gamma_{k,j}| + |\hat{\Gamma}_{k,j}|)$.
\\\\
\emph{Step 5.} For $t \in [mid,bot]$, the situation is as follows. The maps $h_{mid}$ and $h_{bot}$ have already been defined. We interpret these maps as planar maps, identifying the horizontal sections $\mathcal{U}_t$ of the cube $U_{k,j}$ on the domain side with the same square domain which we call $\mathcal{U}$. Both maps $h_{mid}$ and $h_{bot}$ are hence interpreted to be defined on $\mathcal{U}$ and as they have the same boundary map $\varphi_{mid} = \varphi_{bot}$, we may interpret them to map $\mathcal{U}$ into the same target domain $\mathcal{V}$ bounded by the piecewise linear Jordan curve $\varphi_{mid}(\partial \mathcal{U})$. The difference between these two maps is that $h_{mid}$ is defined by the shortest curve extension of $\varphi_{mid}$ and $h_{bot}$ is defined as the shortest curve extension of its boundary values in each of the four child squares of $\mathcal{U}$.

Let us denote by $\mathcal{C}$ the cross formed by the two segments between opposing midpoints of the sides of $\mathcal{U}$. Hence the way $h_{mid}$ maps $\mathcal{C}$ is determined by the shortest curve extension and we denote the image cross by $T_{mid} = h_{mid}(\mathcal{C})$. The way $h_{bot}$ maps $\mathcal{C}$ is predetermined by the piecewise linear approximations of the original boundary map defined in Section \ref{sec:standa}. We denote $T_{bot} = h_{bot}(\mathcal{C})$.

A key point to note is the following. Let $\mathcal{U}'$ denote one of the four children of $\mathcal{U}$. Then we claim that $h_{mid}$ restricted to $\mathcal{U}'$ is actually the shortest curve extension of its boundary value on $\partial \mathcal{U}'$. Let $\ell$ denote one of the horizontal line segments inside $\mathcal{U}'$ (the meaning of 'horizontal' here is as it was used in the definition of the shortest curve extension), with $a$ and $b$ being its endpoints. Then $\ell$ is part of a horizontal segment of $\mathcal{U}$ and is mapped to a curve under $h_{mid}$ which is the shortest such curve between its endpoints. This must mean also that the curve is the shortest curve from $h_{mid}(a)$ to $h_{mid}(b)$ inside $\bar{\mathcal{U}'}$. Moreover, since $h_{mid}$ maps each horizontal segment in $\mathcal{U}$ to its target curve with constant speed, $h_{mid}$ must also have constant speed on $\ell$. This cements the fact that $h_{mid}$ on $\mathcal{U}'$ is the shortest curve extension of its boundary values.

However, the above argument has the following minor defect. In Section \ref{sec:2d}, the shortest curve extension was defined for a boundary map from a square to a piecewise linear Jordan domain. But the map $h_{mid}$ might not map the two line segments making up $\mathcal{C}$ to true Jordan curves as the shortest curve extension may fail to be injective and thus the image cross $T_{mid}$ may touch the boundary in $\bar{\mathcal{V}}$. Nevertheless, these curves are still piecewise linear and are given by a uniform limit of Jordan curves. There is no issue defining the notion of shortest curves and shortest curve extensions to areas bounded by such degenerate Jordan curves as well, and the estimates we have established before in results such as Lemma \ref{lem:sameboundary} and Lemma \ref{lem:changeboundary} extend naturally to this setting as well. This can be seen either by verifying that the proofs go through in the degenerate case as well or use a limiting argument via approximation by actual Jordan curves.

From now the strategy to define a homotopy $h_t$ for $t \in [mid,bot]$ is as follows. For each such $t$, the map $h_t$ on $\partial \mathcal{U}$ will have the same boundary values $\varphi_{mid}$. Moreover, we will define a homotopy of crosses $T_{t}$ between the two crosses $T_{mid}$ and $T_{bot}$. Once such a homotopy has been defined and parametrized as a map $\Phi_t: \mathcal{C} \to T_t$, for each child $\mathcal{U}'$ of $\mathcal{U}$ we define $h_t$ on as the shortest curve extension of its boundary values on $\partial \mathcal{U}'$. Thus $h_t$ will be equal to $\varphi_{mid}$ on $\partial \mathcal{U}$ and to $\Phi_t$ on $\mathcal{C}$.

To construct the homotopy between the two crosses, we would like to apply the same argument from Step 4 which was used to create a homotopy between the crosses depicted in Figure \ref{fig:grid4}. However, in the argument from Step 4 it was essential that the two crosses only had two intersection points (on the curves between $v_1, m_1$ and $v_1, m_2$). In our case, the crosses $T_{mid}$ and $T_{bot}$ may have arbitrarily many intersection points. To address this issue, we define another cross $T_{fix}$ which satisfies this property respective to both the crosses $T_{mid}$ and $T_{bot}$, and then simply deform first $T_{mid}$ to $T_{fix}$ and then to $T_{bot}$. Due to Lemma \ref{lem:sameboundary}, the exact nature of the parametrization $\Phi_t$ does not play a role here and we may assume for example that on each of the four arms of $\mathcal{C}$ the parametrization always has constant speed.

Before defining $T_{fix}$, we make a small modification to $T_{mid}$ in order to replace it with a cross $T_{mid*}$  which does not intersect the boundary except at the four endpoints. Since the cross $T_{mid}$ consists of piecewise linear curves, this modification can be done by moving each of its vertices that touch the boundary (except for the four endpoints) by an arbitrarily small amount towards the interior of $\mathcal{V}$ so that the resulting cross does not intersect itself nor $\partial \mathcal{V}$. This modification provides a homotopy from $T_{mid}$ to $T_{mid*}$ which we may, for example, dedicate the first quarter of the interval $[mid,bot]$ towards in $t$. The fact that this modification to the cross may be done in an arbitrarily small way guarantees that the Lipschitz estimates (in $t$) both on $\mathcal{C}$ and for the shortest curve extensions to the four regions of $\mathcal{V}$ can be controlled by above with a constant of our choice.

It now remains to define $T_{fix}$. Since neither of the crosses $T_{mid*}$ and $T_{bot}$ touch the boundary $\partial \mathcal{V}$ except at their common four endpoints, we may choose $T_{fix}$ for example as follows. We pick a point $P$ in $\mathcal{V}$ close enough to one of the image points of the corners of $\mathcal{U}$ under $\varphi_{mid}$ so that $P$ belongs to $h_{mid*}(\mathcal{U}') \cap h_{bot}(\mathcal{U}')$ for one of the children $\mathcal{U}'$ of $\mathcal{U}$. Then we connect $P$ to the four endpoints of $T_{mid*}$ via piecewise linear curves to form the cross $T_{fix}$. These curves, if chosen to run sufficiently close along the boundary $\partial \mathcal{V}$, may be assumed to satisfy the necessary properties of not intersecting themselves or each other. Moreover, they can be chosen so that two of them intersect $T_{mid*}$ and $T_{bot}$ exactly once and two of them do not intersect these crosses (apart from the endpoints). See Figure \ref{fig:cross1}. This means that the crosses $T_{fix}$ and $T_{mid*}$ are in the same configuration as the crosses in Step 4, and the same goes for $T_{fix}$ and $T_{bot}$. Hence we may repeat the argument to find a homotopy between these crosses, and extend the boundary values defined by this via the shortest curve extension to the whole of $\mathcal{U}$. For each $t$, we lift the copy of $\mathcal{U}$ and the map $h_t$ to the appropriate horizontal section at height $t$ in $U_{k,j}$ and $V_{k,j}$ in order to fully define our extension there.

\begin{figure}[H]
\includegraphics[scale=0.5]{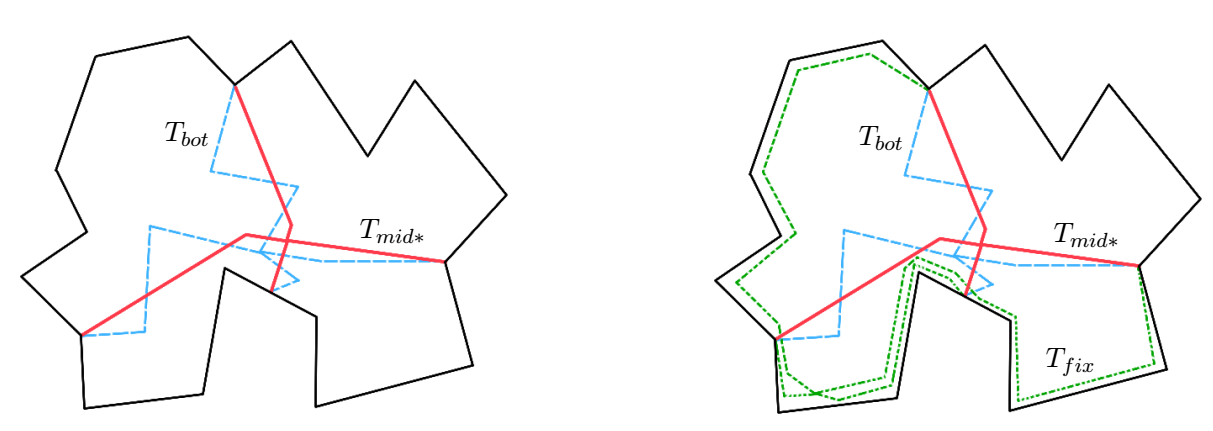}
\caption{Constructing an intermediate cross $T_{fix}$. The original crosses $T_{mid*}$ and $T_{bot}$ are denoted in red and blue color and they intersect a lot. Thus we construct a new intermediate cross $T_{fix}$ denoted in green which does not intersect $T_{mid*}$ and $T_{bot}$ too much.}\label{fig:cross1}
\end{figure}

We have thus defined the extension $h$ as a monotone map on each set $U_{k,j}$ to the image set $V_{k,j}$. We now return to our original goal of controlling the Lipschitz constant of $h$ in $U_{k,j}$ outlined in the beginning of the section. In Step 4, the Lipschitz constant of the boundary value isotopy $\varphi_t$ is controlled by above (in both the space and $t$ variable) by the lengths of the corresponding boundary curves and possibly the lengths of the neighbouring curves. Lemma \ref{lem:changeboundary} then shows that this implies the correct Lipschitz estimates for $h$ in the region where $t \in [top,mid]$. In the region $t \in [mid,bot]$, the map $h$ is defined piecewise as the shortest curve extension yet again, so to obtain the correct Lipschitz estimates one needs only estimate the length of the boundary curves on the image side. These consist of the original boundary curve $\partial \mathcal{V}$ and the lengths of the crosses $T_{mid}$, $T_{fix}$ and $T_{bot}$. The first two can be bounded from above by a constant times the length of $\partial \mathcal{V}$ (which is the length of $\hat{\Gamma}_{k,j}$, while the last one is bounded by the lengths of the image curves of the children $\hat{\Gamma}_{k,j}^{(m)}$. Thus we get the desired estimate that yields a bound on the $W^{1,q}$-norm of $h$ in terms of the quantity on the left hand side of \eqref{eq:discretehomeoextension}.

\end{proof}

\section{Making it all injective}\label{sec:injective}

Let $\varphi : \partial S \to \partial \yy$ be a homeomorphic boundary map to a Jordan domain $\yy$ with piecewise linear boundary. We now describe how to tackle the issue that the shortest curve extension $H_\varphi$ is not injective but rather a monotone map. The main issue is that the images of two horizontal segments $l_{s_1}$ and $l_{s_2}$ of $S$ may intersect each other or intersect the boundary of the image domain $\partial \yy$. However, the saving grace is that these images are shortest curves between their respective endpoints and thus do not cross, allowing us to make a minor modification to the curves so that they do not intersect each other or touch the boundary and therefore create a homeomorphic extension $H_{\varphi}^*$ of $\varphi$. This modification is not too difficult for a single map and was done already in \cite{HP}. However, in our case more details are needed as we need to make this modification consistent in a way that if $\varphi_t$ is a continuous family of boundary maps, not necessarily to the same image domain, then the modified extensions $H_{\varphi_t}^*$ need to be continuous in $t$ and the modification must be done in a way to preserve the Lipschitz estimates in terms of $\varphi_t$.

We consider here the situation where the boundary map $\varphi$ is also piecewise linear. In all of the cases we consider this is true since $\varphi$ is always defined piecewise as a constant speed map. When $\varphi$ and $\partial \yy$ are piecewise linear, it is not difficult to check that then also the shortest curve extension $H_\varphi$ becomes a piecewise linear map on $\bar{S}$.

The aim is to show that the modification from the shortest curve extension $H_\varphi$ to its homeomorphic variant $H_{\varphi}^*$ may be done in an arbitrarily small way in the following sense. On each horizontal segment $l_s$, the map $H_\varphi$ maps $l_s$ to a shortest curve $L_s$ with constant speed. The map $H_\varphi^*$ instead maps $l_s$ to another piecewise linear curve $L_s^*$, also with constant speed, and so that $L_s^*$ may be obtained from $L_s$ by shifting each vertex of $L_s$ by a small distance. We will show that such distances can be chosen to be arbitrarily small, controlled by a single constant per map, which means that the modified map $H_{\varphi_t}^*$ will also be arbitrarily close to $H_\varphi$ which lets us obtain the same Lipschitz-estimates for it.

The idea behind modifying the curves $L_s$ to the curves $L_s^*$ is quite simple. At each vertex of $\partial \yy$ where $L_s$ passes through, we move that vertex of $L_s$ a little bit further away from the boundary. For curves $L_{s'}$ with $s' > s$, this movement should be a little bit larger for vertices on $\partial \yy$ on the image of the part of $\partial S$ below $l_s$ and a little smaller for vertices on $\partial \yy$ on the image of the part of $\partial S$ above $l_s$. See Figure \ref{fig:modify1}. However, in order to make this compatible with a homotopy of boundary maps $\varphi_t$ we must define this process very precisely in order not to have discontinuities in $t$. This is what we now do.

We define a number $D$ as the minimal length between two sides of $\partial \yy$ which are not neighbours. Next, for any point $P \in \partial \yy$ we define the \emph{inner normal} of $P$, denoted $\ell_P$, as the ray which starts from the point $P$, points towards the interior of $\yy$ near $P$, and forms equal angles with $\partial \yy$ i.e. is an angle bisector for the angle of $\partial \yy$ formed at $P$.

For every vertex $P \in \partial \yy$, we pick a positive number $\epsilon_P < 1$ whose role will become apparent later in making the modification process continuous in $t$. We then define the point $V_P$ as the point on $\ell_P$ which is of distance $\epsilon_P D/3$ away from $P$. By the definition of $D$, the point $V_P$ must be at a distance of at least $2D/3$ away from any other side of $\partial \yy$ than the two $P$ lies on. This means that apart from the point $P$, the segment $PV_P$ cannot intersect $\partial \yy$ nor can it intersect any other such segment $QV_Q$ for another vertex $Q$ of $\partial \yy$.

Note that two of the shortest curves $L_s$ may only intersect at points on $\partial \yy$. Since the point $V_P$ is inside $\yy$, for each $P$ there must be a unique parameter $s_P$ for which $L_{s_P}$ passes through $V_P$. We also define $\hat{s}_P$ as the parameter for which $P$ is one of the endpoints of $L_{\hat{s}_P}$. Thus the curves $L_s$ which intersect the segment $PV_P$ are exactly those for which $s \in [s_P,\hat{s}_P]$. It can also be possible that $s_P = \hat{s}_P$, in which case the segment $PV_P$ belongs fully to the curve $L_{s_P}$. This is also the only case in which a curve $L_s$ intersects $PV_P$ more than once. In this case we will not modify the curve $L_{\hat{s}_P}$ which is equivalent with setting $\epsilon_P = 0$.

Suppose that $s_P > \hat{s}_P$. For each $s \in [\hat{s}_P,s_P]$ there is a unique point $X_s$ on $PS_P$ which belongs to $L_s$. Let $f_P : [\hat{s}_P,s_P] \to [0,\epsilon_P D/3]$ denote the function which sends $s$ to $|X_s - P|$. Now $f_P$ is an increasing and surjective piecewise linear function, strictly increasing on the preimage of $(0,\epsilon_P D/3]$, but it is possible that $f_P$ sends a nontrivial interval of parameters $[\hat{s}_P,x]$ to $0$. In fact, this happens exactly in the case where there are multiple curves $L_s$ that intersect at $P$.

The idea now is the following. We pick a strictly increasing surjective piecewise linear function $f_P^* : [\hat{s}_P,s_P] \to [0,\epsilon_P D/3]$ to act as an injective replacement for $f_P$. We wish to make a canonical choice here so for an increasing surjective function $f_P : [0,1] \to [0,1]$ for which $f^{-1}(\{0\}) = [0,A]$ we set
\[f_P^*(x) =
\left\{
	\begin{array}{ll}
		x/(2A)  & \mbox{when } x \in [0,A], \\
		(f_P(x) + 1)/2 & \mbox{otherwise}.
	\end{array}
\right.\]
The way we will modify each curve $L_s$ for $s \in [\hat{s}_P,s_P]$ is by moving the point $X_s$ on $L_s$ to a new point $X_s^*$ on $PV_P$ so that $|X_s^* - P| = f_P^*(s)$.

If $s_P < \hat{s}_P$, we do the exact same process as above only on the interval $[s_P,\hat{s}_P]$ on which the analogously defined function $f_P$ will be decreasing instead of increasing. Similarly we choose $f_P^*$ as a strictly decreasing function.

We now define the curves $L_s^*$. For each curve $L_s$, we make note of all the segments $PV_P$ which this curve passes through. We only consider segments with $s_P \neq \hat{s}_P$ as to neglect cases where the segment $PV_P$ is fully on $L_s$. On each of the applicable segments $PV_P$ intersecting $L_s$ we move the point $X_s$ on the curve $L_s$ to $X_s^*$. 
Note that the curves $L_{s_P}$ and $L_{\hat{s}_P}$ are unchanged with respect to this process (although they may be changed on other segments $QV_Q$).\\\\

\begin{figure}[H]\label{fig:modify1}
\includegraphics[scale=0.5]{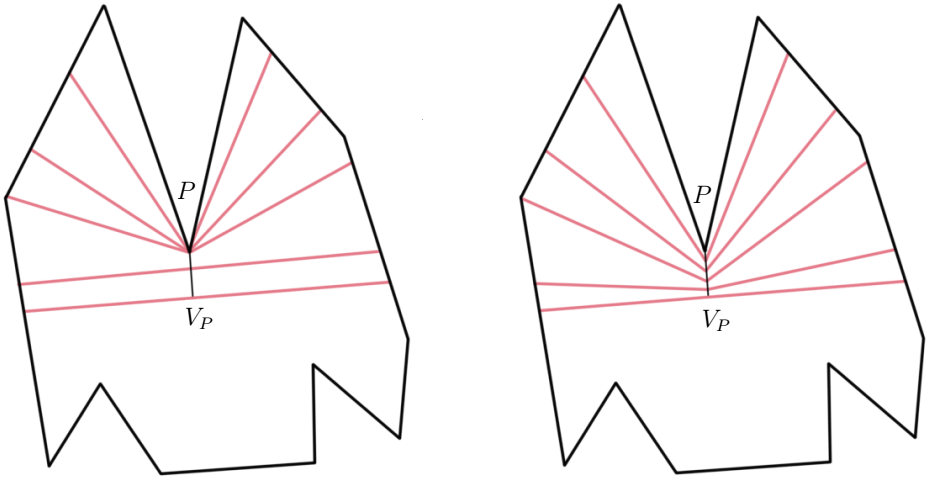}
\caption{Modifying the curves $L_s$ on the segment $PV_P$.}
\end{figure}
\noindent
\emph{Step 1.} Proving that the curves $L_s^*$ do not intersect $\partial \yy$ except at their endpoints.\\\\
Fix $s$ and consider the curve $L_s$. For each vertex $P$ of $\partial \yy$, we consider the segments $P V_P$ which are disjoint. Considering the intersection points of $L_s$ with all such segments $PV_P$, this splits the curve $L_s$ into segments $Q_0 Q_1, Q_1 Q_2, \ldots , Q_{N-1}Q_N$ so that $Q_0, Q_N$ are the endpoints of $L_s$ and for each $Q_j$, there is a point $P_j$ which is a vertex of $\partial \yy$ so that $Q_j \in P_j V_{P_j}$. Moreover, we assume that there are no other such points on $L_s$.

Consider now a segment $Q_j Q_{j+1}$ with $0 < j < N-1$. During the deformation from $L_s$ to $L_s^*$, the point $Q_j$ is moved on the segment $P_j V_{P_j}$ to another point $Q_j^*$. Suppose for the contrary that the segment $Q_j^* Q_{j+1}^*$  intersects the boundary $\partial \yy$. Let $Q_j^r = (1-r)Q_j + rQ_j^*$. As neither $Q_j^*$ or $Q_{j+1}^*$ intersect $\partial \yy$, there must be a minimal number $0 < r < 1$ so that $Q_j^r Q_{j+1}^r$ intersects $\partial \yy$. We now consider two cases:
\begin{enumerate}
\item If a vertex $P$ of $\partial\yy$ intersects $Q_j^r Q_{j+1}^r$. Basic geometry dictates that such a vertex $P$ cannot share a side with $P_j$ or $P_{j+1}$. If $P$ equals $Q_j^r$ or $Q_{j+1}^r$, this contradicts the definition of $D$ as then the distance from $P$ to either $P_j$ or $P_{j+1}$ would be too small, seeing as $|Q_j^r - P_j| \leq D/3$ holds for all $j$ and $r$ due to $Q_j^r \in P_jV_{P_j}$. If $P$ is strictly between $Q_j^r$ and $Q_{j+1}^r$, then again a simple geometrical argument shows that there must be a non-endpoint of $Q_j Q_{j+1}$ which is on $PV_P$, a contradiction with the definition of the points $Q_j$.
\item If a point $X$ of $\partial \yy$ which is not a vertex intersects $Q_j^r Q_{j+1}^r$. We obtain a similar contradiction as above if $X$ is either of $Q_j^r$ or $Q_{j+1}^r$. In the case where $X$ is strictly inside $Q_j^r Q_{j+1}^r$, the segment of $\partial \yy$ on which $X$ is on must be parallel to $Q_j^r Q_{j+1}^r$. But for any two segments which are parallel and intersect each other, one must contain an endpoint of the other one. Thus this reduces to one of the cases already considered.
\end{enumerate}
\emph{Step 2.} Proving that the curves $L_s^*$ do not intersect each other.\\\\
If two of the curves $L_s^*$ and $L_{s'}^*$ intersected each other with $s < s'$. Then for all $r \in (s,s')$ the curve $L_r^*$ would also necessarily intersect both $L_s^*$ and $L_{s'}^*$ or either it would provide a separation between them. But for $r$ close enough to $s$, the curves $L_r^*$ and $L_s^*$ may not intersect. This is due to the fact that these curves may be decomposed into the same number of segments $I_j^r$ and $I_j^s$, $j = 1,\ldots,N$, and so that $I_j^r \to I_j^s$ as $r \to s$. This convergence implies that for $r$ close enough to $s$, the segment $I_j^r$ cannot intersect $I_{j'}^s$ unless $j' \in \{j-1,j,j+1\}$. However, even in this case these segments may not intersect due to geometrical reasons, as the nature of the construction guarantees that $I_j^r$ and $I_j^s$ do not intersect.\\\\
\emph{Step 3.} Uniform estimates in $t$.\\\\
During the construction made in Section \ref{sec:3d}, we have created an extension $h:[0,1]^3 \to [0,1]^3$ of the boundary map $\varphi$ so that each level $[0,1]^2 \times \{t\}$ is mapped to $\R^2 \times \{t\}$ . For each $t$, such a level is divided into a number (depending on $t$) of dyadic squares whose boundaries are mapped to piecewise linear Jordan curves by $h$ on the target side. Moreover, inside these squares the map $h$ is defined by the shortest curve extension of its boundary values. For each dyadic level, there is a specific parameter $t$ at which the construction changes from being based on those dyadic squares to being based on their children. The exact behaviour of $h$ at this parameter was described in Step 5 of Section \ref{sec:3d} at the parameter $t = mid$ in the cube $U_{k,j}$. We let the sequence of such parameters be denoted by $t_1 > t_2 > t_3 > \ldots$ corresponding to each dyadic level.

We first describe how to modify the extension $h$ inside each interval $I_j = (t_{j+1},t_j]$ without paying mind to the continuity between successive intervals. We focus now on a single dyadic square $\tilde{Q}_{k,j} \times \{t\}$ on the domain side and its target set, which we interpret as a planar Jordan domain $\yy_t$ with piecewise linear boundary. We may apply continuity and the fact that there is an upper bound on the number of vertices of each piecewise linear curve to deduce that the quantity $D$ as defined earlier on $\yy_t$ has a uniform lower bound for $t \in I_j$. Here the quantity $D$ and all other quantities introduced in the earlier description of the construction need to be interpreted as functions of $t$.

We now appeal to the behaviour of the piecewise linear curve $\partial \yy_t$. In a neighbourhood of parameters $t$ where the number of vertices of $\partial \yy_t$ is constant the domain $\yy_t$ changes in $t$ only by moving these vertices around in a continuous way. There is hence a correspondence between the segments $PV_P$ in $t$ in this neighbourhood and thus a necessary step to guarantee continuity of the modified extension is to ensure that the length of each such segment is a continuous function in $t$. This length of $PV_P$ was defined as $\epsilon_P D/3$. Since $D$ is locally bounded from below in $t$, $\epsilon_P$ can be chosen for each $t$ in such a way as to make $\epsilon_P D$ a continuous function in $t$ in such a neighbourhood. In fact, we choose $\epsilon_P D$ to be a piecewise linear function to maintain Lipschitz-continuity in $t$ as well (we pay proper attention to estimates later). Another case to account for are shortest curves $L_s$ which completely contain a segment $PV_P$. This happens exactly when the shortest curve $L_{\hat{s}_P}$ with endpoint $P$ bisects the angle of the boundary at $P$. In such a case no other curve $L_s$ may pass through $P$, which allows us to set $\epsilon_P = 0$ at any parameter $t$ where this happens without losing injectivity. This can be done while maintaining the continuity of $\epsilon_P D$ in $t$, for example by multiplying an already chosen function $\epsilon_P(t)$ with a (piecewise linear) function $G(t)$ for which $G(t) \in [0,1)$ and $G(t) = 0$ exactly for those parameters $t$ for which $L_{\hat{s}_P}$ contains $PV_P$.

The number of vertices of $\partial \yy_t$ does not generally remain constant, as there may be new vertices appearing from an edge turning into two edges via a new angle being created at a given point $P$ on that edge. The reverse may also happen to reduce the vertex count by one, but for the purposes of proving continuity both of these cases are symmetric to each other. Let us hence assume that at time $T_0$ the point $P = P(T_0)$ lies on an edge of $\partial \yy_{T_0}$, but on the interval $(T_0,T_1)$ the point $P(t)$ is a true vertex of $\partial \yy_t$. In this case we do as before on $(T_0,T_1)$, choosing $\epsilon_P D$ to be continuous in terms of $t$. Moreover, we choose $\epsilon_P$ in such a way that $\epsilon_P D \to 0$ as $t \to t_0$. This means that the segment $PV_P$ shrinks to a point as $t \to T_0$, which guarantees continuity at this point also.

For a fixed parameter $t$, it is clear that as the numbers $\epsilon_P$ are chosen uniformly small enough, for example, by multiplying each with a small constant $\delta_t > 0$ independent of $P$, the modified extension $H^*_\varphi$ is arbitrarily close to the original extension $H_\varphi$ in the Lipschitz norm. Moreover as the quantities $\epsilon_P D$ were chosen to be Lipschitz continuous, choosing $\delta_t$ as a piecewise linear function in $t$ with small enough Lipschitz norm guarantees that the map $(z,t) \to H^*_{\varphi_t}(z)$ may be chosen arbitrarily close to the original map $h$ in the Lipschitz norm for $t \in (t_{k+1},t_k]$. This shows that the Lipschitz estimates obtained in the previous section may be inherited by the modified extension as well.

Finally, we address the case of the parameters $t_k$ where we switch from one dyadic level to another ($t = mid$ in $U_{k,j}$). We pick a parameter $t_k^* < t_k$ slightly below $t_k$ so that on the level $t_k^*$ the extension $h$ is given by the shortest curve extension in the four dyadic children instead. Choosing $t_k^*$ close enough to $t_k$ lets us assume that the two maps levels $t_k^*$ and $t_k$ are arbitrarily close to each other in the Lipschitz norm. Moreover, due to this we may assume that the two modified maps are also as close in the Lipschitz norm as we want. For the sake of this argument we interpret these modified maps as planar maps $h_{t_k}, h_{t_k^*} : S \to \yy$ from a square to a piecewise Lipschitz Jordan domain, and recall that they have the same boundary values. As both of these maps are piecewise linear and homeomorphic, for $t_k^*$ close enough to $t_k$ we may assume that each of the maps $h^{(\tau)} := (1-\tau) h_{t_k} + \tau h_{t_k^*}$ is also homeomorphic for $\tau \in [0,1]$ due to the fact that the Jacobian determinant of $h^{(\tau)}$ must be bounded away from zero for all $\tau$ when $h_{t_k}$ and $h_{t_k^*}$ are close enough in the Lipschitz norm.

We may then redefine the extension for parameters $t \in [t_k^*,t_k]$ by setting it equal to $h^{(\tau)}$ for $\tau = (t-t_k)/(t_k^*-t_k)$. Note that the Lipschitz norm in $t$ may now be very large here due to the fact that the denominator $t_k^* - t_k$ may be arbitrarily small. To fix this, we rescale the parametrization on the interval $(t_{k+1},t_k]$ on the domain and target side so that if $M$ denotes the midpoint of this interval, we scale $(t_k^*,t_k]$ to $(M,t_k]$ and $(t_{k+1},t_k^*]$ to $(t_{k+1},M]$. The length of the interval $(M,t_k]$ is hence comparable to $2^{-k}$, which means that the Lipschitz constant of the map for parameters $t \in (M,t_k]$ on $U_{k,j}$ is controlled by $2^k\hat{\Gamma}_{k,j}$ as we have wanted. This finishes the construction and the proof.

\end{document}